\documentclass{amsart}
\usepackage{amssymb} 
\usepackage{amsbsy}
\usepackage{amscd}
\usepackage{amsmath}
\usepackage{amsthm}
\usepackage{amsxtra}
\usepackage{latexsym}
\usepackage[mathscr]{eucal}
\usepackage{upgreek}
\usepackage{wasysym}
\usepackage[all,cmtip]{xy}

\usepackage{xcolor}
\definecolor{rouge}{rgb}{0.85,0.1,.4}
\definecolor{bleu}{rgb}{0.1,0.2,0.9}
\definecolor{violet}{rgb}{0.7,0,0.8}
\usepackage[dvipdfmx,colorlinks=true,linkcolor=bleu,urlcolor=violet,citecolor=rouge]{hyperref}
\newcommand{\cmu}{{\check\mu}}
\newcommand{\nc}{\newcommand}
\nc{\We}[2]{\mathbb{V}^{#1}_{#2}}
\nc{\Si}[2]{\mathbb{L}^{#1}_{#2}}
\nc{\mi}{\varphi}
\nc{\ppart}{(\!(t)\!)}
\nc{\al}{\alpha}
\nc{\ka}{\kappa}
\nc{\LP}{{}^L\neg P}
\nc{\n}{{\mathfrak n}}
\nc{\ghat}{\wh{\g}}
\def\neg{\negthinspace}

\renewcommand{\L}{\mathbb{L}}

\newcommand{\wh}{\widehat}

\newcommand{\bra}{{\langle}}
\newcommand{\ket}{{\rangle}}

\newcommand{\Lam}{\Lambda}

\newcommand{\on}{\operatorname}
\newcommand{\+}{\mathop{\oplus}}
\renewcommand{\*}{{\otimes}}
\newcommand{\mc}{\mathcal}
\newcommand{\mf}{\mathfrak}

\newcommand{\g}{\mf{g}}
\newcommand{\h}{\mf{h}}

\newcommand{\affg}{\widehat{\mf{g}}}

\newcommand{\isomap}{{\;\stackrel{_\sim}{\to}\;}}
\newcommand{\Z}{\mathbb{Z}}
\newcommand{\C}{\mathbb{C}}

\newcommand{\Q}{\mathbb{Q}}
\newcommand{\W}{\mathscr{W}}
\newcommand{\ra}{\rightarrow}
\newcommand{\lam}{\lambda}

\newcommand{\voa}{vertex operator algebra}

\def\leq{\leqslant}
\def\geq{\geqslant}

\DeclareMathOperator{\End}{End}

\DeclareMathOperator{\gr}{gr}

\DeclareMathOperator{\ad}{ad}
\DeclareMathOperator{\Hom}{Hom}

\theoremstyle{theorem}
\newtheorem{Th}{Theorem}[section]
\newtheorem{MainTh}{Theorem}
\newtheorem{Pro}[Th]{Proposition}

\newtheorem{Lem}[Th]{Lemma}

\newtheorem{Co}[Th]{Corollary}

\theoremstyle{remark}

\newtheorem{Rem}[Th]{Remark}
\newtheorem{Conj}{Conjecture}

\newtheorem{Ex}[Th]{Example}

\title[Urod algebras and Translation of W-algebras]{Urod algebras and Translation  of  W-algebras}
\author{Tomoyuki Arakawa}
\address{Research Institute for Mathematical Sciences, Kyoto University,
Kyoto 606-8502 JAPAN}
\email{arakawa@kurims.kyoto-u.ac.jp}
\author{Thomas Creutzig}
\address{University of Alberta
Department of Mathematical and Statistical Sciences
Edmonton, AB T6G 2G1, Canada}
\email{creutzig@ualberta.ca}
\author{Boris Feigin}
\address{National Research University Higher School of Economics, 101000, Myasnitskaya ul.~20, Moscow, Russia, and Landau Institute for Theoretical Physics, 142432, pr.\ Akademika Semenova 1a, Chernogolovka, Russia}
\email{bfeigin@gmail.com}

\begin{document}
\maketitle

\begin{abstract}
In this work, we  introduce Urod algebras associated to simply-laced Lie algebras as well as the concept of translation of $W$-algebras. 

Both results are achieved by 
showing that the quantum Hamiltonian reduction
commutes with tensoring with integrable representations, 
that is,
for $V$ and $L$ an affine vertex algebra and an integrable affine vertex algebra associated with $\g$,
 we have the vertex algebra isomorphism
$H_{DS,f}^0(V\otimes L)\cong H_{DS,f}^0(V)\otimes L$
where in the left-hand-side the Drinfeld-Sokolov reduction is taken with respect to the diagonal action of $\affg$ on $V\* L$.

The proof  is based on some new construction
of automorphisms of vertex algebras, which may be of independent interest. As corollaries we get fusion categories of modules of many exceptional $W$-algebras and we can construct various corner vertex algebras. 

A major motivation for this work is that Urod algebras of type $A$ provide a representation theoretic interpretation
 of the celebrated Nakajima-Yoshioka blowup equations
 for the moduli space of framed torsion free sheaves on $\mathbb{CP}^2$  of  an arbitrary rank.

%
\end{abstract}

\section{Introduction}
In \cite{BerFeiLit16}
Bershtein, Litvinov and the third named author introduced the
{\em Urod algebra},
which gives a representation theoretic interpretation
 of the celebrated {\em Nakajima-Yoshioka blowup equations} \cite{Nakajima2005} for the moduli space of framed torsion free sheaves on $\mathbb{CP}^2$  of  rank two,
 via the Alday-Gaiotto-Tachikawa (AGT) correspondence \cite{AGT}.
One of the aims of the present paper is to introduce
the {\em higher rank Urod algebras},
which generalizes the result of \cite{BerFeiLit16} to the the case of sheaves at arbitrary rank.

In fact,
it turned out in recent works (see e.g. \cite{FGuk,CG}) that
Urod algebras  
appear not only in the  AGT correspondence but also  in various  theories of vertex algebras in connection with
higher dimensional quantum field theories. 
This work provides the first systematic study of Urod algebras appearing in  various contexts.

Another aim of this paper is to introduce the {\em translation for (affine) $W$-algebras}.
Namely,  we show that
for any integrable highest  representation $L$ of level $\ell\in \Z_{>0}$
gives rise to an exact functor
\begin{align}
T_L: \W^k(\g,f)\on{-Mod}\ra \W^{k+\ell}(\g,f)\on{-Mod},\quad M\mapsto M\* L,
\label{eq:translation-by-L}
\end{align}
where $\W^k(\g,f)$ is the $W$-algebra associated with $\g$ and its nilpotent element $f$ at level $k\in \C$
(\cite{FF90,KacRoaWak03}).

We  establish these results
by showing that 
the quantized Drinfeld-Sokolov reduction commutes with tensoring with integrable representations.

\smallskip

Let us describe our results in more details.

\subsection{Main Theorem}
Let
$\g$ be a simple Lie algebra,
$\affg=\g((t))\+\C K$ be the affine Kac-Moody Lie algebra
associated with $\g$ defined by the commutation relation
$$[xf,yg]=[x,y]fg+(x,y)\on{Res}_{t=0}(gdf)K,$$
$[K,\affg]=0$,
where $(~|~)$ is the normalized inner product of $\g$ (it is $1/2h^{\vee}$ times the Killing form of $\g$).
For $k\in \C$,
let
$V^k(\g)=U(\affg)\*_{U(\g[[t]]\* \C K)}\C$,
where $\C$ is regarded as a one-dimensional representation of 
$\g[[t]]\* \C K$ on which $\g[[t]]$ acts trivially and $K$ acts 
via multiplication by $k$.
$V^k(\g)$ is naturally a vertex algebra, and is called the {\em universal affine vertex algebra}
associated with $\g$ at level $k$.

Any (graded) quotient $V$ of $V^k(\g)$ inherits the vertex algebra structure from $V^k(\g)$.
Let $\L_k(\g)$
be the  unique simple (graded) quotient 
  of $V^k(\g)$.
The vertex algebra $\L_k(\g)$ is 
integrable as an $\affg$-module if and only if $k\in \Z_{\geq 0}$.

For a nilpotent element $f$ of $\g$,
let
 $H_{DS,f}^\bullet(M)$ be the BRST cohomology of the quantized Drinfeld-Sokolov reduction
 associated with $(\g,f)$
with coefficients in a $\affg$-module $M$ (\cite{FF90,KacRoaWak03}).
The $W$-algebra associated with $(\g,f)$ at level $k$ is by definition
the vertex algebra
\begin{align*}
\W^k(\g,f)=H_{DS,f}^0(V^k(\g)).
\end{align*}
For any smooth $\affg$-module $M$
of level $k$,
$H_{DS,f}^i(M)$, $i\in \Z$,
is a module over $\W^k(\g,f)$.
More generally,
for any vertex algebra $V$ equipped with a vertex algebra homomorphism
$V^k(\g)\ra V$ and a $V$-module $M$,
$H_{DS,f}^i(M)$, $i\in \Z$,
is a module over the vertex algebra $H_{DS,f}^0(V)$.

\begin{MainTh}\label{MainTh:iso}
Let $V$ be a quotient of the universal affine vertex algebra $V^k(\g)$ and $\ell\in \Z_{\geq 0}$.
We have a vertex algebra isomorphism
\begin{align}
H_{DS,f}^0(V\otimes L_{\ell}(\g))\cong H_{DS,f}^0(V)\otimes  L_{\ell}(\g)
\label{eq:trans-iso}
\end{align}
where in the left-hand-side the Drinfeld-Sokolov reduction is taken with respect to the diagonal action of $\affg$ on $V\*  L_{\ell}(\g)$.
More
generally,
let $V$ be a vertex algebra 
equipped with a vertex algebra homomorphism
$V^k(\g)\rightarrow V$.
Then we have an isomorphism
\begin{align}
H_{DS,f}^\bullet(V\otimes L_{\ell}(\g))\cong H_{DS,f}^\bullet(V)\otimes  L_{\ell}(\g),
\label{eq:trans-iso}
\end{align}
of graded vertex algebras,
and
 for any $V$-module $M$,
$ \L_{\ell}(\g)$-module $N$,
there is an isomorphism
\begin{align*}
H_{DS,f}^\bullet(M\*N)\cong H_{DS,f}^\bullet(M)\* N.
\end{align*}
as modules over 
$H_{DS,f}^\bullet(V\otimes L)\cong H_{DS,f}^\bullet(V)\otimes L$.
\end{MainTh}
In the case that
$\g=\mf{sl}_2$ 
and $L=\L_1(\mf{sl}_2)$,
the isomorphism \eqref{eq:trans-iso} 
was established  in \cite{BerFeiLit16}.
Our argument only requires certain properties of integrable representations and especially also works for superalgebras, see section \ref{sec:super}.

We note that
the existence of the isomorphism \eqref{eq:trans-iso}
as vector spaces
is not difficult to see. 
However,
it is not a priori clear at all
why there should exist an isomorphism of vertex algebras.
We also note that
\eqref{eq:trans-iso}  
is not compatible with 
the standard conformal gradings of both sides.
To remedy this,
we need to change the conformal vector
of $\L_{\ell}(\g)$ on the right-hand-side
to a new conformal vector $\omega_{Urod}$,
which we call the {\em Urod conformal vector} (see Section \ref{section:Urod conformal vector}).

The proof  of Theorem \ref{MainTh:iso}  is based on some new construction
of automorphisms of vertex algebras, which may be of independent interest,
see  Section \ref{section:auto}
for the details.

\subsection{Translation for $W$-algebras}
By applying Theorem \ref{MainTh:iso}
to $V=V^k(\g)$,
$k\in \C$,
we obtain the vertex algebra isomorpshim
$$H_{DS,f}^0(V^k(\g)\* L_{\ell}(\g))\cong \W^k(\g,f)\* \L_{\ell}(\g).$$
Consequently, 
 the natural vertex algebra homomorphism
$V^{k+\ell}(\g)\ra V^k(\g)\* \L_{\ell}(\g)$
induces a vertex algebra homomorpshim
\begin{align*}
\W^{k+\ell}(\g,f)\ra H_{DS}^0(V^k(\g)\* \L_{\ell}(\g))\overset{\sim}{\longrightarrow}  \W^k(\g,f)\* \L_{\ell}(\g)
\end{align*}
Therefore,
for any $\W^k(\g,f)$-module $M$ and
 any integrable representation $L$ of 
$\affg$ of level $\ell$,
$M\*L$ is has the structure of an $\W^{k+\ell}(\g,f)$-module.
As a consequence,
we obtain
{\em the translation by $L$},
that is, the exact functor
\eqref{eq:translation-by-L}
 as we wished.

Recall that the Zhu algebra
of $\W^k(\g,f)$ is isomorphic to
the finite $W$-algebra \cite{Pre02} 
$U(\g,f)$ 
associated with $(\g,f)$ (\cite{Ara07,De-Kac06}).
Also,
the
Zhu algebra of $\L_{\ell}(\g)$ is isomorphic to
the quotient $U_{\ell}(\g)$  of 
$U(\g)$ by the two-sided ideal generated by
$e_{\theta}^{\ell+1}$ (\cite{FreZhu92}),
and so is that of
$\L_{\ell}(\g)$ with the Urod conformal structure (\cite{A2012Dec}).
Therefore, by taking the Zhu algebras
we obtain from 
\eqref{eq:translation-by-L}
an algebra homomorphism
\begin{align}
U(\g,f)\ra U(\g,f)\* U_{\ell}(\g).
\label{eq:Zhu-algebra}
\end{align}
Since any finite-dimensional $\g$-module
is an $U_{\ell}(\g)$-module for a sufficiently large $\ell$,
\eqref{eq:Zhu-algebra}
gives 
$M\* E$
 a structure of 
$U(\g,f)$-module
for any $U(\g,f)$-module $M$
and a finite-dimensional $\g$-module $E$.
We expect  that this  $U(\g,f)$-module structure of $M\* E$ does not depend on the choice of  a sufficiently large $\ell$,
and coincides with the one obtained by Goodwin \cite{Goo11}.

\subsection{Higher rank Urod algebras}
We denote also by $\W^k(\g)$
the $W$-algebra $\W^k(\g,f_{prin})$ associated with a principal nilpotent element $f_{prin}$
of $\g$.
Let $\g$ be simply-laced
and suppose that $k+h^{\vee}-1\not\in \Q_{\leq 0}$,
where $h^{\vee}$ is the dual Coxeter number of $\g$.
By the coset construction \cite{ACL19}
of the principal $W$-algebra
$\W^k(\g)$,
we have a conformal vertex algebra embedding
\begin{align}
V^k(\g)\* \W^{\ell}(\g)\hookrightarrow  V^{k-1}(\g)\* \L_1(\g),
\label{eq:GKO-embedding}
\end{align}
where $\ell$ is the number defined by the formula
\begin{align}
\frac{1}{k+h^{\vee}}+\frac{1}{\ell+h^{\vee}}=1.
\label{eq:def-of-ell}
\end{align}
By taking the Drinfeld-Sokolov reduction with respect to the level $k$ action of $\affg$,
\eqref{eq:GKO-embedding} gives rise to the full vertex algebra embedding
\begin{align}
\W^k(\g,f)\* \W^{\ell}(\g)\hookrightarrow  H_{DS,f}(V^{k-1}(\g)\* \L_1(\g))\cong \W^{k-1}(\g,f)\* 
\mc{U}(\g),
\label{eq:Urod-embedding}
\end{align}
where the last isomorphism follows from Theorem \ref{MainTh:iso}
and
$\mc{U}(\g)=\mc{U}(\g,f)$ is the vertex algebra
  $\L_1(\g)$  equipped with the Urod conformal vector $\omega_{Urod}$. 
We call the vertex operator algebra 
$\mc{U}(\g)$
 the  {\em Urod algebra}.
 In the case that $\g=\mf{sl}_2$
 and $f=f_{prin}$,
$\mc{U}(\g)$ is exactly the Urod algebra 
 introduced in
  \cite{BerFeiLit16}.

 Let $L$ be a level one integrable representation of $\affg$,
 which is naturally a  module over the Urod algebra $\mc{U}(\g)$.
By \eqref{eq:Urod-embedding},
for any $\W^k(\g,f)$-module $M$,
the tensor product 
$M\*L$ has the structure of a
$\W^k(\g,f)\* \W^{\ell}(\g)$-module.
We are able to describe
the decomposition of 
$M\*L$ as $\W^k(\g,f)\* \W^{\ell}(\g)$-modules
for various $M$ and $L$
(Theorems \ref{Th:generic-decom}, \ref{Th:dec:generic-principal},
\ref{Th:dec-of-Verma},
\ref{Th:dec-admi-pri}
and Corollaries \ref{Th:dec-of-Verma}, \ref{Co:Urod-decomposition}).

In particular,
Corollary  \ref{Th:dec-of-Verma}
states that
when $M$ is a generic Verma module
of $\W^{k-1}(\g)$,
then $M\* L$  decomposes into a  direct sum of tensor products 
of Verma modules of $\W^{k}(\g)$ and $\W^{\ell}(\g)$.
In the case that $\g=\mf{sl}_2$
with an appropriate choice of $L$,
this provides the decomposition that was
%
used in 
\cite{BerFeiLit16} to give a
representation theoretic interpretation 
of the Nakajima-Yoshioka blowup equations
for
the 
moduli space of rank two framed torsion free  sheaves on $\mathbb{CP}^2$.
In a forthcoming paper we show 
how the decomposition 
for $\g=\mf{sl}_n$ stated in Corollary  \ref{Th:dec-of-Verma}
can be used to give a
representation theoretic interpretation 
of Nakajima-Yoshioka blowup equations
for the moduli space of  framed rank $n$ sheaves 
on $\mathbb{CP}^2$
via the AGT conjecture established by Schiffmann and Vasserot \cite{Schiffmann2013}.

\subsection{Higher rank Urod algebras and $\on{VOA}[M_4]$} The Urod algebra is proposed to be important for general smooth $4$-manifolds. 
This appeared in the recent  work \cite{FGuk}
of Gukov and the third named author.
For a compact simply-laced Lie group $G$,
one can conjecturally \cite{FGuk} associate a vertex operator algebra $\on{VOA}[M] =\on{VOA}[M, G]$
 to every smooth $4$-manifold $M$
 and an cateogry of modules for  $\on{VOA}[M]$
 to every boundary component of $M$. The vertex algebra $\on{VOA}[M]$ should then act on the cohomology of the moduli space of $G$-instantons on $M$. 
 Moreover, the invariant $\on{VOA}[M]$ 
 should have the following property: 
Glueing $4$-manifolds along a common boundary  amounts to extending the tensor product of the two associated vertex algebras along the categories of modules,
see  \cite{CKM2} for 
the theory of these  vertex algebra extensions.

When $G=SU(2)$
we have \cite{FGuk} that
\begin{align*}
\on{VOA}[M_4\# \overline{\mathbb{C} \mathbb{P}^2}]=\mc{U}(\mf{sl}_2)\* \on{VOA}[M_4],
\end{align*}
where $M_4\# \overline{\mathbb{CP}^2}$ is  the connected sum of $M_4$ and $\overline{\mathbb{C} \mathbb{P}^2}$.
One expects the same type of formula for any simply laced $G$ with $\mc{U}(\mf{sl}_2)$ replaced by the corresponding higher rank Urod algebra. 
%

  \subsection{Higher rank Urod algebras and Vertex algebras for S-duality}
  The present work was originally motivated by
  a conjecture that appeared in the context of vertex algebras for $S$-duality \cite{CG} .

  The problem lives inside four-dimensional supersymmetric GL-twisted gauge theories and vertex algebras appear on the intersection of three-dimensional topological boundary conditions. 
  Such vertex algebras are typically constructed out of $W$-algebras and affine vertex algebras associated to the Lie algebra $\g$ of the gauge group $G$, and the coupling constant $\Psi$ relates to the level shifted by the dual Coxeter number $h^\vee$. Different boundary conditions can be concanated to yield other boundary conditions and corresponding vertex algebras are related via vertex algebra extensions. Most of \cite{CG} is dealing with simply-laced $\g$ and 
the discussion at the bottom of page 22 of \cite{CG} is concerned with the concanation of boundary conditions called $B_{1, 0}, B_{-1, 1}$ and $B_{0, 1}$.
The main expectation is that the resulting corner vertex algebra coincides with the corner vertex algebra between the boundary conditions $B_{1, 0}$ and $B_{0, 1}$ dressed by extra decoupled degrees of freedom corresponding to $\L_1(\g)$. In vertex algebra language this expectation is precisely the statement of Theorem \ref{Th:generic-decom}
with $\mu=0=\nu$.

There are further conjectures around vertex algebras and $S$-duality. These are mainly the construction of junction or corner vertex algebras which are typically large extensions of products of two vertex algebras associated to $\g$. 
Theorem \ref{Th:extension},
which  gives a lattice type construction of vertex operator algebras
of CFT type using $W$-algebras in place of Heisenberg algebras,
proves them  in a series of cases.

%
%
%
%
%

\subsection{Rigidity  of vertex tensor categories}
One of the most difficult problems of the subject of vertex algebras is the understanding of tensor categories of modules of a given vertex algebra. 
The theory  of tensor categories of modules of  vertex algebras
    has been developed in the series of papers \cite{HuaLep94, HuaLep95}.
 In particular,    
Yi-Zhi Huang has shown the existence of vertex tensor categories for lisse vertex algebras 
 without the rationality assumption \cite{Huang09}.

The most challenging technical problem 
for vertex tensor categories
is  to prove the rigidity of objects.
This is crucial as rigidity gives the categories substantial structure and many useful theorems only hold for rigid tensor categories. 
While
this problem was settled by
 Huang for rational, lisse vertex algebras \cite{Hua08rigidity},
 it is wide open for non-rational lisse vertex algebras.
 
 In fact,  it is expected that the 
tensor categories of modules 
should exist for much more general vertex algebras.
A strong evidence was given in \cite{CreHuaYan18}
that showed the category of ordinary modules over an admissible affine vertex algebra
associated with a simply laced Lie algebra
has the structure of a vertex tensor category.

We conjecture that the category of ordinary modules
over a quasi-lisse vertex algebra  \cite{AraKaw18} has the structure of a vertex tensor category (Conjecture \ref{Conj:verte-tensor}).
Note that an admissible affine vertex algebra is quasi-lisse.

Assuming this conjecture,
and using an idea of \cite{Cre19}, we use
the decomposition stated in Theorem \ref{Th:generic-decom}
to prove 
that certain categories of quasi-lisse $W$-algebras
   at admissible level are fusion (Theorem \ref{Th::tensor-cateogry}).
Since the conjecture is valid for lisse $W$-algebras,
this gives a strong evidence for the rationality conjecture   \cite{KacWak08,Ara09b}  of lisse $W$-algebras
at admissible levels,
which has been settled 
only in some special  cases  \cite{A2012Dec,AEkeren19}.

It seems that the translation functor is a good tool to prove rationality of vertex operator algebras in suitable cases. We will explain in forthcoming work how to employ the translation functor in order to get new rational $\W$-algebras at non-admissible levels associated to Deligne's exceptional series \cite{ACK}.

\vspace{2mm}

\noindent {\bf Acknowledgements} 
T. A. is partially supported by by  by JSPS KAKENHI Grant Numbers 17H01086, 17K18724,
19K21828.
T.C. appreciates discussions with 
Sergei Gukov and Emanuel Diaconescu. T.C. is supported by NSERC $\#$RES0048511.


\section{A construction of automorphisms of vertex algebras}\label{section:auto}
Let $V$ be a vertex algebra.
As usual,
we set $a_{(n)}=\on{Res}_{z=0}z^n a(z)$
for $a\in V$,
where $a(z)$ is the quantum field corresponding to $a$.
We have
\begin{align}
[a_{(m)},b_{(n)}]=\sum_{j\geq 0}\begin{pmatrix}
m\\ j\end{pmatrix}(a_{(j)}b)_{(m+n-j)}
\label{eq:commutator-formula}
\end{align}
for $a,b\in V$, $m,n\in \Z$.
In particular,
$[a_{(0)},b_{(n)}]=(a_{(0)}b)_{(n)}$.

Let $A$ be an element of $V$ such that 
its zero mode $A_{(0)}$ acts semisimply on $V$ 
so that $V=\bigoplus_{\lam\in \C 
}V[\lam]$,
where 
$V[\lam]=\{v\in V\mid A_{(0)}v=\lam v\}$.
We assume that $A\in V[0]$.
Set $V[\geq \lam]=\bigoplus\limits_{\mu\in  \lam
+\mathbb{R}_{\geq 0}} V[\mu]\supset
V[> \lam]=\bigoplus\limits_{\mu> \lam} V[\mu]$.

Suppose that there exists another element
$\hat{A}\in V[\geq 0]$
such that $\hat{A}\equiv A \pmod{V[>0]}$
and $\hat{A}_{(0)}$ acts locally finitely on $V$.
Then,
$\hat{A}_{(0)}$ acts semisimply on $V$,
and
for $v\in V[\lam]$,
 there exist a unique
eigenvector
$\tilde{v}$ of $\hat{A}_{(0)}$
eigenvalue $\lam$ such that 
$\tilde{v}\equiv v\pmod{V[>\lam]}$.
By extending this correspondence linearly,
we obtain
a linear map
\begin{align}
V\ra V,\quad v\mapsto \tilde{v}.
\label{eq:automorphism}
\end{align}
\begin{Lem}\label{Lem:auto}
The map \eqref{eq:automorphism}
is an automorphism of $V$.
\end{Lem}
\begin{proof}
Clearly, \eqref{eq:automorphism} is a linear isomorphism.
We wish to show that 
 \eqref{eq:automorphism} is a homomorphism 
 of vertex algebras.
 It is clear that $\widetilde{|0\ket}=|0\ket$.
Let $v\in V[\lam]$,
$w\in V[\mu]$,
$n\in \Z$.
By \eqref{eq:commutator-formula}, $\tilde{v}_{(n)}\tilde{w}$
is an eigenvector
of  $\tilde{A}_{(0)}$ of eigenvalue $\lam+\mu$
and
$\tilde{v}_{(n)}\tilde{w}\equiv v_{(n)}w\pmod{V[>\lam+\mu]}$.
Therefore,
$\tilde{v}_{(n)}\tilde{w}=\widetilde{v_{(n)}w}$.
Similarly,
$T\tilde{v}$
is an  eigenvector
of  $\tilde{A}_{(0)}$ of eigenvalue $\lam$
such that 
$T\tilde{v}\equiv Tv \pmod{V[>\lam]}$.
Hence $T\tilde{v}=\widetilde{Tv}$.
This completes the proof.
\end{proof}
We also have the following.
\begin{Lem}\label{lem:can replace}
The action of $\hat{A}_{(0)}$ on $V$ coincides with 
that of $\tilde{A}_{(0)}$.
\end{Lem}
\begin{proof}
For an eigenvector $v\in V[\geq \lam]$ of 
$\hat{A}_{(0)}$  of eigenvalue $\lam$,
$(\hat{A}_{(0)}-\tilde{A}_{(0)})v$ is also 
an eigenvector of $\hat{A}_{(0)}$  of eigenvalue $\lam$.
On the other hand,
 $(\hat{A}_{(0)}-\tilde{A}_{(0)})v$
 belongs to $V[>\lam]$.
 Since all eigenvalues of $\hat{A}_{(0)}$ 
 on $V[>\lam]$ are greater than $\lam$,
 the vector  $(\hat{A}_{(0)}-\tilde{A}_{(0)})v$
 must be zero.
 This completes the proof.
\end{proof}

Let $M$ be a $V$-module on which 
both $A_{(0)}$ and $\hat{A}_{(0)}$ act semisimply.
Set
$M[\lam]=\{m\in M\mid 
A_{(0)}m=\lam m\}$,
$M[>\lam]=\sum_{\mu>\lam}M[\mu]$.
We can define a linear isomorphism
\begin{align}
M\isomap M, \quad m\mapsto \tilde{m},
\label{eq:module-map}
\end{align}
that sends 
$m\in M[\lam]$
to 
a unique
eigenvector
$\tilde{m}$ of $\hat{A}_{(0)}$
of eigenvalue $\lam$ such that 
$\tilde{m}\equiv m\pmod{M[>\lam]}$.
The following assertion can be shown in the same manner as 
Lemma \ref{Lem:auto}.
\begin{Lem}\label{Lem:iso2}
We have $\widetilde{a_{(n)}m}=\tilde{a}_{(n)}\tilde{m}$
for $a\in V$, $m\in M$, $n\in \Z$,
that is,
\eqref{eq:module-map} is an isomorphism
of $V$-modules.
\end{Lem}

\section{Preliminaries on Drinfeld-Sokolov reduction}
Let $\g$ be a simple Lie algebra over $\C$,
and let $V^k(\g)$ be the universal affine vertex algebra associated with
$\g$ at level $k$ as in the introduction.
A $V^k(\g)$-module is the same as a smooth module $M$ of level $k$ over the affine Kac-Moody algebra $\affg$.
Here a $\affg$-module $M$ 
 is called {smooth}
if 
$x(z)=\sum_{n\in \Z}(xt^n)z^{-n-1}$ is a (quantum) field on $M$
for all $x\in \g$,
that is,
$(xt^n)m=0$ for a sufficiently large $n$
for any $m\in M$.

Let 
$f$ be a nilpotent element of $\g$.
Let
\begin{align}
\g=\bigoplus_{j\in \frac{1}{2}\Z}\g_j
\label{eq:grading}
\end{align}
be a good grading (\cite{KacRoaWak03}) of $\g$ for $f$,
that is,
$f\in \g_{-1}$,
$\ad f:\g_j\ra \g_{j-1}$ is injective for $j\geq 1/2$
and surjective for $j\leq 1/2$.
Denote by
$x_0$ the semisimple element of $\g$ that defines the grading,  i.e.,
\begin{align}
\g_j =\{x\in \g\mid
 [x_0,x] = jx\}. 
 \label{eq:x0}
\end{align}
We write $\deg x=d$ if $x\in \g_d$.

Let $\{e,h,f\}$ be an $\mf{sl}_2$-triple associated with $f$ in $\g$.
Then the grading defined by $x_0=1/2h$ is good,
and is called a {\em Dynkin grading}.

Fix a Cartan subalgebra $\h$ of $\g$ that is contained in the Lie subalgebra $\g_0$.
Let
$\Delta$ be the set of roots of $\g$,
$\g=\h\+ \bigoplus_{\alpha\in \Delta}\g_{\alpha}$ the root space decomposition.
Set $\Delta_j=\{\alpha\in \Delta\mid \g_{\alpha}\subset \g_j\}$,
so that $\Delta=\bigsqcup_{j\in \frac{1}{2}\Z}\Delta_j$.
Put $\Delta_{>0}=\bigsqcup_{j>0}\Delta_j$.
Let $I=\{1,2,\dots, \on{rk}\g\}$,
and
let $\{x_a\mid a\in I\sqcup \Delta \}$
be a basis of $\g$ such that $x_{\alpha}\in \g_{\alpha}$,
$\alpha\in \Delta$,
and
$x_i\in \h$, $i\in I$.
Denote by $c_{a b}^d$ the corresponding structure constant.

Set  $\g_{\geq 1}=\bigoplus_{j\geq 1}\g_j$,
$\g_{>0}=\bigoplus_{j\geq 1/2}\g_j$,
and
let  $\chi:\g_{\geq 1}\ra \C$
be the character defined by
$\chi(x)=(f|x)$.
We extend $\chi$ to the character $\hat{\chi}$ of 
$\g_{\geq 1}[t,t^{-1}]$ by setting
$\hat\chi(xt^n)= \delta_{n,-1}\chi(x)$.
Define $$F_{\chi}=U(\g_{>0}[t,t^{-1}])\otimes_{U(\g_{>0}[t]+\g_{\geq 1}[t,t^{-1}])}\C_{\hat\chi},$$
where $\C_{\hat \chi}$ is the one-dimensional representation of 
$\g_{>0}[t]+\g_{\geq 1}[t,t^{-1}]$
on which $\g_{\geq 1}[t,t^{-1}]$ acts by  the character $\hat \chi$
and  $\g_{>0}[t]$ acts trivially.
Since it is a smooth $\g_{>0}[t,t^{-1}]$-module,
the space $F_{\chi}$ is a module over the  vertex subalgebra $V(\g_{>0 })\subset V^k(\g)$
generated by $x_{\alpha}(z)$ with $\alpha\in \Delta_{>0}$.
For $\alpha\in \Delta_{> 0}$,
let $\Phi_{\alpha}(z)$ denote the image of $x_{\alpha}(z)$ in 
$(\End F_{\chi})[[z,z^{-1}]]$.
Then
\begin{align*}
\Phi_{\alpha}(z)=\chi(x_\alpha)\quad \text{for  }\alpha\in \Delta_{\geq 1}
\end{align*}
and
\begin{align*}
\Phi_{\alpha}(z)\Phi_{\beta}(w)\sim \frac{\chi([x_{\alpha},x_{\beta}])}{z-w}.
\end{align*}
There is a unique vertex algebra structure
on $F_{\chi}$
such that $|0\ket=1\* 1$ is the vacuum vector and 
\begin{align*}
Y((\Phi_{\alpha})_{(-1)}|0\ket ,z)=\Phi_{\alpha}(z)
\end{align*}
for $\alpha\in \Delta_{>0}$.
(Note that $(\Phi_{\alpha})_{(-1)}|0\ket=\chi(x_{\alpha})|0\ket $ for $\alpha\in \Delta_{\geq 1}$.)
In other words,
$F_{\chi}$ has the structure of the $\beta\gamma$-system 
associated with the symplectic vector space $\g_{1/2}$
with the symplectic form
\begin{align}
\label{eq:symp-form-on}
\g_{1/2}\times \g_{1/2}\ra \C,\quad
(x, y)\mapsto  \chi([x,y]).
\end{align}

Next, let $\bigwedge^{\infty/2+\bullet}(\g_{>0})$ be the vertex superalgebra
generated by odd fields
$\psi_{\alpha}(z)$,
$\psi_{\alpha}^*(z)$,
$\alpha\in \Delta_{>0}$,
with the OPEs
\begin{align*}
\psi_{\alpha}(z)\psi_{\beta}^*(z)\sim \frac{\delta_{\alpha,\beta}}{z-w},
\quad \psi_{\alpha}(z)\psi_{\beta}(z)\sim
\psi_{\alpha}^*(z)\psi_{\beta}^*(z)\sim 0
\end{align*}
Let 
$\bigwedge^{\infty/2+\bullet}(\g_{>0})=\bigoplus_{n\in \Z}
\bigwedge^{\infty/2+n}(\g_{>0})$ be the
$\Z$-gradation 
defined by
$\deg |0\ket =0$,
$\deg (\psi_\alpha)_{(n)}=-1$,
$\deg (\psi_\alpha^*)_{(n)}=1$.

For a smooth $\affg$-module $M$,
set
\begin{align}
C(M):=M\* F_{\chi}\* \bigwedge\nolimits^{\infty/2+\bullet}(\g_{>0}).
\end{align}
Then $C(M)=\bigoplus\limits_{i\in \Z}C^i(M)$,
$C^i(M)=M\* F_{\chi}\* \bigwedge\nolimits^{\infty/2+i}(\g_{>0})$.
Note that $C(V^k(\g))$ is naturally a vertex superalgebra
and
$C(M)$ is a module over the vertex superalgebra
$C(V^k(\g))$
for any smooth $\affg$-module $M$.
Define 
\begin{align*}
Q(z)=\sum_{\alpha\in \Delta_{>0}}x_\alpha(z)\psi_{\alpha}^*(z)
+\sum_{\alpha\in \Delta_{>0}}\Phi_{\alpha}(z)\psi_{\alpha}^*(z)
-\frac{1}{2}
\sum_{\alpha,\beta,\gamma\in \Delta_{>0}}c_{\alpha,\beta}^\gamma
\psi_{\alpha}^*(z)\psi_{\beta}^*(z)\psi_{\gamma}(z),
\end{align*}
where 
we have omitted the tensor product symbol.
Then $Q(z)Q(w)\sim 0$ and we have
$Q_{(0)}^2=0$
on any $C(V^k(\g))$-module.
The cohomology
\begin{align*}
H^\bullet_{DS,f}(M):
=H^{\bullet}(C(M),Q_{(0)})
\end{align*}
is called the BRST cohomologyof the Drinfeld-Sokolov reduction
associated with $f$ with coefficients in $M$
(\cite{FF90,KacRoaWak03}, see also \cite{Ara05}).
By definition \cite{Feu84},
$H^\bullet_{DS,f}(M)$
is   the semi-infinite $\g_{>0}[t,t^{-1}]$-cohomology 
 $H^{\frac{\infty}{2}+\bullet}(\g_{>0}[t,t^{-1}],M\otimes  F_{\chi})$
with coefficients in the diagonal $\g_{>0}[t,t^{-1}]$-module $M\otimes  F_{\chi}$.

The vertex algebra 
$$\W^k(\g,f):=H^0_{DS,f}(V^k(\g))$$
is called the $W$-algebra associated with $(\g,f)$ at level $k$,
which is 
  conformal provided that $k\ne -h^{\vee}$.
The vertex algebra structure
of $\W^k(\g,f)$ does not depend on the choice of a good grading (\cite{AKM}),
however, its conformal structure does.
The central charge of $\W^k(\g,f)$ is given by
\begin{align}
\dim \g-\frac{1}{2}\dim \g_{1/2}-12|\frac{\rho}{\sqrt{k+h^{\vee}}}-\sqrt{k+h^{\vee}}x_0|^2,
\end{align}
where $\rho$ is the half sum of positive roots of $\g$.

Let $\W_k(\g,f)$ be the unique simple graded quotient of $\W^k(\g,f)$.

Let $\on{KL}$ be the full subcategory 
of $\affg$-modules
 consisting of objects
on which $\g$ acts  semisimply
and $t\g[t]$ acts locally nilpotently,
and let $\on{KL}_k$
be the full subcategory of $\on{KL}$ 
consisting of modules of level $k$.

Let $Q$ be the root lattice of $\g$, 
$\check{Q}$ the coroot lattice, 
$P$ the weight lattice,
$\check{P}$ the coweight lattice,
$P_+$  the set of dominant integral weights,
$\check{P}_+$  the set of dominant integral coweights of $\g$.
For $\lam\in P_+$,
set
\[
\mathbb{V}^k_\lambda := U(\widehat\g) \otimes _{U(\g[t]\oplus \C K)} E_\lambda\in \on{KL}_k,
\]
where $E_\lambda$ is the irreducible finite-dimensional $\g$-module of highest weight $\lambda$ lifted to a $\g[t]$-module 
by letting $\g[t]t$ act trivially and $K$ by multiplication with the scalar $k$. 
Note that $V^k(\g)=\mathbb{V}^k_0$ as a $\affg$-module.
We denote by $\Si{k}{\lam}$ the unique simple quotient of $\mathbb{V}^k_\lambda$.
More generally,
for any weight $\lam$ of $\g$
we denote by $\Si{k}{\lam}$ the irreducible highest weight representation of $\affg$ with highest weight
$\lam$ and level $k$.

\begin{Th}[\cite{FreGai07,Ara09b}]\label{Th:vanishing}
We have
$H_{DS,f}^i(M)=0$ for  $i\ne 0$ and $M\in \on{KL}_k$.
In particular,
the functor
\begin{align*}
\on{KL}_k\ra \W^k(\g,f)\on{-Mod},\quad M\mapsto H_{DS,f}^0(M),
\end{align*}
is exact.
\end{Th}

Note that
for
$M\in \on{KL}_k, N\in \on{KL}_{\ell}$,
we have $M\*N\in \on{KL}_{k+\ell}$.
Therefore,
$H^i_{DS,f}(M\* N)=0$ for $i\ne 0$.
In particular,
$H^i_{DS,f}(M\* L)=0$ for $i\ne 0$
if $M\in\on{KL}$ and 
$L$ is an integrable representation of $\affg$.


%
%
%

%
%
%
%
%
%

\section{Proof of  Theorem \ref{MainTh:iso}}\label{section:proof}
Let $k\in \C$,
$\ell\in \Z_{\geq 0}$,
and
set
\begin{align*}
C:=C(V^k(\g)\* \L_{\ell}(\g))=V^k(\g)\* \L_{\ell}(\g)\* 
F_\chi\*\bigwedge\nolimits^{\infty/2+\bullet}(\g_{>0}).
\end{align*}
For $t\in \C$, define
the element $Q_t\in C$ by
\begin{align}
Q_t(z)=&
\sum_{\alpha\in \Delta_{>0}}(\pi_1(x_\alpha)(z)+t^{2\alpha(x_0)}\pi_2(x_\alpha)(z))\psi_{\alpha}^*(z)\\
&
+\sum_{\alpha\in \Delta_{>0}}\Phi_{\alpha}(z)\psi_{\alpha}^*(z)
-\frac{1}{2}
\sum_{\alpha,\beta,\gamma\in \Delta_{>0}}c_{\alpha,\beta}^\gamma
\psi_{\alpha}^*(z)\psi_{\beta}^*(z)\psi_{\gamma}(z),
\nonumber
\end{align}
where 
$\pi_1(x_a)(z)$ (resp.\ $\pi_2(x_a)(z)$)
denotes the action of $x_a(z)$, $a\in I\sqcup \Delta$, on 
$V^k(\g)$ (resp.\ on $\L_{\ell}(\g)$).
Then
$Q_t(z)Q_t(w)\sim 0$,
and therefore,
$(Q_t)_{(0)}^2=0$.
It follows that
$(C,(Q_t)_{(0)})$ is a differential graded vertex algebra,
and the corresponding cohomology
$H^{\bullet}(C,(Q_t)_{(0)})$ is naturally a vertex algebra.
Clearly,
\begin{align}
&H^{i}(C,(Q_{t=0})_{(0)})\cong H_{DS,f}^i(V^k(\g))\* \L_{\ell}(\g)
=\delta_{i,0}\W^k(\g,f)\* \L_{\ell}(\g),\\
&H^{i}(C,(Q_{t=1})_{(0)})\cong H_{DS,f}^i(V^k(\g)\* \L_{\ell}(\g))
=\delta_{i,0}H_{DS,f}^0(V^k(\g)\* \L_{\ell}(\g)),
\end{align}
see Theorem \ref{Th:vanishing}.

By \cite[3.7]{Ara05},
the differential $(Q_t)_{(0)}$ decomposes  as
  \begin{align}
(Q_t)_{(0)}=d_t^{st}+d^{\chi},
\quad (d_t^{st})^2=(d^{\chi})^2=\{d_t^{st},d^{\chi}\}=0,
\label{eq:dec-diff}
\end{align}
where 
\begin{align}
\nonumber  &d_t^{st}=\sum_{\alpha\in \Delta_{>0}}\sum_{n\in \Z}(\pi_1(x_{\alpha})_{(n)}+t^{2\alpha(x_0)}\pi_2(x_{\alpha})_{(n)})\psi_{\alpha,(-n)}^*\\
&\qquad\qquad\quad+\sum_{\alpha\in \Delta_{1/2}}\sum_{n< 0}\Phi_{\alpha,(n)}\psi^*_{\alpha,(-n)}\\
&\qquad \qquad \quad -\frac{1}{2}
\sum_{\alpha,\beta,\gamma\in \Delta_{>0}}c_{\alpha,\beta}^\gamma
\sum_{m,n\in \Z}
\psi_{\alpha,(m)}^*\psi_{\beta,(n)}^*\psi_{\gamma,(-m-n)},
\nonumber
\\
&d^{\chi}=\sum_{\alpha\in \Delta_{1/2}}\sum_{n\geq 0}\Phi_{\alpha,(n)}\psi^*_{\alpha,(-n)}
+\sum_{\alpha\in \Delta_1}\chi(x_{\alpha})\psi^*_{\alpha,(0)}.
\end{align}

Define the Hamiltonian $H$
on $C$ by 
\begin{align*}
H=H_{standard}^{V^k(\g)}+(\omega_{\L_{\ell}(\g)})_{(1)}
+(\omega_{F_{\chi}})_{(1)}+(\omega_{\bigwedge\nolimits^{\infty/2+\bullet}(\g_{>0})})_{(1)}
-\pi_1(x_0)_{(0)}-\pi_2(x_0)_{(0)},
\end{align*}
where 
$H_{standard}^{V^k(\g)}$
is the standard Hamiltonian of $V^k(\g)$ that 
gives 
$\pi_1(x)$, $x \in \g$,
conformal weight  one,
$\omega_{\L_{\ell}(\g)}$ is the Sugawara conformal vector 
of $\L_{\ell}(\g)$,
\begin{align*}
&\omega_{F_{\chi}}(z)=\frac{1}{2}\sum\limits_{\alpha\in \Delta_{1/2}}:\partial_z\Phi^{\alpha}(z) \Phi_{\alpha}(z):,\\
&\omega_{\bigwedge\nolimits^{\infty/2+\bullet}(\g_{>0})}(z)=
\sum_{j>0}\sum_{\alpha\in \Delta_{j}}j:\psi_{\alpha}^*(z)\partial \psi_{\alpha}(z):
+\sum_{j>0}\sum_{\alpha\in \Delta_{j}}(1-j):(\partial \psi_{\alpha}^*(z))\psi_{\alpha}(z):.
\end{align*}
Here,
$\{\Phi^{\alpha}\}$ is a dual basis
to $\{\Phi_{\alpha}\}$,
that is,
$\Phi^{\alpha}(z)\Phi_{\beta}(w)\sim \delta_{\alpha \beta}/(z-w)$.
Then $[H, d_t^{st}]=0=[H,d_t^{\chi}]$,
and thus,
$H$ defines a Hamiltonian on
the vertex algebra  $H^{\bullet}(C,(Q_t)_{(0)})$.


Following \cite{FreBen04,KacRoaWak03},
define
 \begin{align*}
J^a(z)={\pi}_1(x_a)(z)+\sum_{\beta,\gamma\in \Delta_+}c_{a,\beta}^\gamma :\psi_{\gamma}(z)\psi_{\beta}^*(z):
\end{align*}
for $a\in I\sqcup \Delta$.
We also 
 denote by
$J^x$ the linear combination
of $J^a$,
$a\in I\sqcup \Delta_{\leq 0}$,
corresponding to $x\in \g_{\leq 0}:=\bigoplus_{j\leq 0}\g_j$.
Let $C_{\leq 0}$ be the vertex subalgebra
of $C$ generated by
$J^x(z$) ($x\in \g_{\leq 0}$),
$\pi_2(x)(z)$ ($x\in \g$),
$\psi^*_{\alpha}(z)$ ($\alpha\in \Delta_{>0}$),
$\Phi_\alpha(z)$
($\alpha\in \Delta_{1/2}$),
and 
let $C_{> 0}$ be the vertex subalgebra
of $C$ generated by
$\psi_{\alpha}(z)$
and
$$((Q_t)_{(0)}\psi_\alpha)(z)=J^\alpha(z)+t^{\alpha(h)}\pi_2(x_{\alpha})(z)+\Phi_{\alpha}(z)$$
($\alpha\in \Delta_{>0}$).

As in \cite{FreBen04,KacRoaWak03},
we find that
both $C_{\leq 0}$
and $C_{> 0}$
are subcomplexes of $(C, (Q_t)_{(0)})$
and that
$C\cong C_{\leq 0}\* C_{>0}$ as 
complexes.
Moreover,
we have $$H^i(C_{>0},(Q_t)_{(0)})=\begin{cases}\C&\text{for }i= 0\\
0&\text{for }i\ne 0,\end{cases}$$ and therefore,
\begin{align}
H^{\bullet}(C, (Q_t)_{(0)})\cong H^\bullet(C_{\leq 0},(Q_t)_{(0)})
\end{align}
as vertex algebras.
Since the cohomological gradation takes only non-negative values on $C_{\leq 0}$,
it follows that
$H^0(C,(Q_t)_{(0)})= H^0(C_{\leq 0},(Q_t)_{(0)})$
is a vertex subalgebra of $C_{\leq 0}$.

Note that the vertex algebra $C_{\leq 0}$ does not depend on the parameter $t\in \C$.
Also,
$C_{\leq 0}$ is preserved by the action of both
$d_t^{st}$ and $d^{\chi}$.

Let $C_{\leq 0,\Delta}=C_{\leq 0}\cap C_{\Delta}$,
so that 
$C_{\leq 0}=\bigoplus_{\Delta}C_{\leq 0,\Delta}$.
\begin{Lem}
For each $\Delta$,
$C_{\leq 0,\Delta}$ is a finite-dimensional subcomplex
of $C_{\leq 0}$.
\end{Lem}
\begin{proof}
The generators
$J^x(z)$ ($x\in \g_{\leq 0}$),
$\psi^*_{\alpha}(z)$ ($\alpha\in \Delta_{>0}$),
and
$\Phi_\alpha(z)$
($\alpha\in \Delta_{1/2}$) 
have positive conformal weights
with respect to the Hamiltonian $H$.
Therefore,
it is sufficient to show that
the vertex subalgebra
$\L_{\ell}(\g)$
that is generated by
$\pi_2(x)(z)$, $x\in \g$,
has finite-dimensional
weight spaces
$\L_{\ell}(\g)_{\Delta}:=\L_{\ell}(\g)\cap C_{\leq 0,\Delta}$
and the conformal weights of $L_{\ell}(\g)$
is bounded from the above.
On the other hand,
the action of  $H$
on $\L_{\ell}(\g)$ is the same as the twisted action
of $(\omega_{\L_{\ell}(\g)})_{(1)}$
corresponding to Li's delta operator
associated with $-x_0$.
Since $\L_{\ell}(\g)$ is rational and lisse,
 the conformal weights of
this twisted representation 
are bounded from above
and the weight spaces are finite-dimensional.
 This completes the proof.
\end{proof}

Since 
both
$d_t^{st}$ abd $d^{\chi}$
preserve  $C_{\leq 0,\Delta}$,
we can consider 
the spectral seqeunce
$E_r\Rightarrow H^{\bullet}(C_{\leq 0})$
such that 
$d_0=d^{\chi}$
and $d_1=d_t^{st}$,
which is converging since 
each $C_{\leq 0,\Delta}$ is finite-dimensional.
As in \cite{FreBen04,KacWak04,KacWak05},
we find that
\begin{align}
E_1^{\bullet,q}
=H^q(C_{\leq 0},d^{\chi})=\delta_{q,0}V^{k^{\natural}}(\g^f)\* \L_{\ell}(\g),
\end{align}
where 
$V^{k^{\natural}}(\g^f)$ is the 
vertex subalgebra
of 
$C_{\leq 0}$ generated by
$J^x(z)$, $x\in \g^f$.
Thus,
 the spectral sequence collapses at $E_1=E_{\infty}$,
and we get the vertex algebra isomorphism
\begin{align}
\gr H^q(C_{t,\leq 0})\cong \delta_{q,0}V^{k^{\natural}}(\g^f)\* \L_{\ell}(\g).
\label{eq:gr-HCt}
\end{align}
Here $\gr H^q(C_{t,\leq 0})$
is the associated  graded vertex algebra
with respect to the 
filtration that defines the spectral sequence.
By \eqref{eq:gr-HCt}, 
for each $v\in V^{k^{\natural}}(\g^f)\* \L_{\ell}(\g)$
there exists a cocycle
\begin{align*}
\hat{v}=v_0+v_1+v_2+\dots
\quad\text{(a finite sum)}
\end{align*}
such that
\begin{align*}
v_0=v,\quad d^{st}v_i=-d^{\chi}v_{i+1}.
\end{align*}
Set \begin{align}
A:=\pi_2(x_0)\in \L_{\ell}(\g),
\label{eq:A-for-translation}
\end{align}
where we recall that $x_0$ is defined by  \eqref{eq:x0}.
Then $C_{\leq 0}=\bigoplus_{\lam\in \frac{1}{2}\Z}C_{\leq 0}
[\lam]$,
$C_{\leq 0}[\lam]=\{c\in C_{\leq 0}\mid A_{(0)}c=\lam c\}$,
and $A\in C_{\leq 0}[0]$.
Consider the corresponding cocycle
$\hat{A}$ that has cohomolological degree zero.
Since $d^{st}_tA\subset C_{\leq 0}[>0]$
and $d^{\chi}C_{\leq 0}[\lam]\subset C_{\leq 0}[\lam]$,
we may assume that
$\hat{A}\equiv A\pmod{C_{\leq 0}[>0]}$.
Moreover, we may assume that 
$\hat{A}$ is homogenous with respect to the
Hamiltonian $H$, 
so that 
$\hat{A}$ preserves each 
finite-dimensional
subcomplex $C_{\leq 0,\Delta}$.
In particular,
$\hat{A}$  is locally finite on $C_{\leq 0}$.
Therefore, we can apply the construction
of Section \ref{section:auto}
to obtain
the automorphism
\begin{align}\label{eq:auto1}
\mi_t:C_{\leq 0}\isomap C_{\leq 0},\quad c\mapsto \tilde{c}
\end{align}
of vertex algebras.

Note that
$\tilde{A}\equiv A\pmod{C[>0]}$ and
the operator $A_{(0)}$ acts 
semisimply
on the whole complex $C$,
so that 
$C=\bigoplus_{\lam\in \frac{1}{2}\Z}C[\lam]$,
$C[\lam]=\{c\in C\mid A_{(0)}c=\lam c\}$.
Since $\mi_t$ is an automorphism of the vertex algebra,
$\tilde{A}\in \mi_t(\L_{\ell}(\g))$
generates a Heisenberg vertex subalgebra.
In particular,
$\tilde{A}_{(0)}$ acts semisimply on $C$.
Therefore,
by replacing $\hat{A}$ by $\tilde{A}$ (see 
Lemma \ref{lem:can replace}),
we can extend the automorphism
\eqref{eq:auto1} to the automorphism
\begin{align}\label{eq:auto2}
C\isomap C\quad c\mapsto \tilde{c},
\end{align}
which is also denote by $\mi_t$.

\begin{Pro}\label{Pro:Qnew=Qt}
We have
$\mi_t(Q_{t=0})_{(0)}=( Q_{t})_{(0)}$ on $C$.
In particular,
$\mi_{t}$ defines an isomorphism
\begin{align*}
(C,(Q_{t=0})_{(0)})\cong (C,(Q_{t})_{(0)})
\end{align*}
of differential graded vertex algebras.
\end{Pro}
\begin{proof}
We first show that
$\mi_t(Q_{t=0})_{(0)}=( Q_{t})_{(0)}$ on $C_{\leq 0}$.
Since
$\mi_t(Q_{t=0})_{(0)}\tilde{A}=\mi_t(Q_{t=0})_{(0)}\mi_t(A)=
\mi_t((Q_{t=0})_{(0)}A)=0
$,
we have
$[\mi_t(Q_{t=0})_{(0)},\tilde{A}_{(0)}]=0$.
Also,
$[(Q_t)_{(0)},\tilde{A}_{(0)}]=
[(Q_t)_{(0)},\hat{A}_{(0)}]=0$
on $C_{\leq 0}$ by 
Lemma \ref{lem:can replace}.
Let $c\in C_{\leq 0}[\geq \lam]$
be an eigenvector of 
$\tilde{A}_{(0)}$ of eigenvalue $\lam$.
Then
the vector
$(\mi_t(Q_{t=0})_{(0)}-(Q_t)_{(0)})c$ is also an eigenvector
of  $\tilde{A}_{(0)}$ of eigenvalue $\lam$.
On the other hand, 
note that 
$\mi_t(Q_{t=0})\equiv Q_{t=0}\equiv Q_{t=1}\pmod{C_{\leq 0}[\geq \lam]}$.
Hence,
$(\mi_t(Q_{t=0})_{(0)}-(Q_t)_{(0)})c\in C_{\leq 0}[>\lam]$.
Since all the eigenvalues of $\tilde{A}_{(0)}$
on $C_{\leq 0}[>\lam]$ are greater than $\lam$,
we get that $(\mi_t(Q_{t=0})_{(0)}-(Q_t)_{(0)})c=0$. 
Hence, 
$\mi_t(Q_{t=0})_{(0)}=(Q_t)_{(0)}$ on $C_{\leq 0}$.

Next, since 
$(Q_t)_{(0)}\tilde{A}=0$,
we have
$[(Q_t)_{(0)},\tilde{A}_{(0)}]=0$ on the whole space  $C$.
Therefore
we  can repeat the same argument for 
 an eigenvector $c\in C[\geq \lam]$ of 
$\tilde{A}_{(0)}$ of eigenvalue $\lam$
to obtain that 
$(\mi_t(Q_{t=0})_{(0)}-(Q_t)_{(0)})c=0$
for all $c\in C$.

The 
last assertion follows 
since
$\mi_t$ preserves the cohomological gradation
as $\hat{A}$ has cohomological degree zero.

\end{proof}

By Proposition \ref{Pro:Qnew=Qt},
$\mi:=\mi_{t=1}$ defines an isomorphism
\begin{align}
(C,(Q_{t=0})_{(0)})\cong (C,(Q_{t=1})_{(0)})
\end{align}
of differential graded vertex algebras.
We have shown the following assertion.
\begin{Th}
The automorphism
$\mi$ 
induces an isomorphism
$H^{\bullet}(C,(Q_{t=0})_{(0)})\cong H^{\bullet}(C,(Q_{t=1})_{(0)})$.
In particular,
\begin{align*}
\W^k(\g,f)\* \L_{\ell}(\g)\cong H_{DS,f}^0(V^k(\g)\* \L_{\ell}(\g))
\end{align*}
as vertex algebras.
\end{Th}

\begin{proof}[Proof of Theorem \ref{MainTh:iso}
]
Let $V$ be a vertex algebra equipped with a 
vertex algebra homomorphism $V^k(\g)\ra V$.
Since $\L_{\ell}(\g)$ is rational,
both $A_{(0)}$ and $\tilde{A}_{(0)}$
acts semisimply on 
$C(V\* \L_{\ell}(\g))$.
Hence we can apply Lemma \ref{Lem:auto}
to extend $\varphi$ to the isomorphism
$C(V\* \L_{\ell}(\g), (Q_{t=0})_{(0)})\cong C(V\* \L_{\ell}(\g), (Q_{t=t})_{(0)})$
of differential graded vertex algebras,
which gives the 
 the isomorphism
$H_{DS,f}^\bullet(V\* \L_{\ell}(\g))\cong H_{DS,f}^\bullet(V)\* \L_{\ell}(\g)$
of vertex algebras.
Similarly,
if $M$ is a $V$-module and $N$ is an $\L_{\ell}(\g)$-module,
it follows from Lemma \ref{Lem:iso2}
that we have an isomorphism
$H_{DS,f}^\bullet(M\* N)\cong H_{DS,f}^\bullet(M)\* N$
as 
modules over
$H_{DS,f}^\bullet(V\* \L_{\ell}(\g))\cong H_{DS,f}^\bullet(V)\* \L_{\ell}(\g)$.

\end{proof}

\begin{Ex}\label{Ex:sl2}
Let $\g=\mf{sl}_2=\on{span}_{\C}\{e,h,f\}$ and $\ell=1$,
and consider the Drinfeld-Sokolov reduction associated with $f$.
Then $C=V^k(\g)\* L_1(\g)\*
\bigwedge\nolimits^{\infty/2+\bullet}(\g_{>0})$,
and $\bigwedge\nolimits^{\infty/2+\bullet}(\g_{>0})$
is generated by odd fields $\psi(z)=\psi_{\alpha}(z)$, $\psi^*(z)
=\psi_{\alpha}^*(z)$.
We have
$$Q_t(z)=(e_1(z)+t^2 e_2(z))\psi^*(z)+\psi^*(z),$$
where we have set
 $x_i=\pi_i(x)$ for $x\in \g$,
 $i=1,2$.
 The vertex subalgebra $C_{\leq 0}$ is 
 generated by 
 $J^{f_1}(z)=f_1(z)$,
 $J^{h_1}(z)=h_1(z)+2:\!\psi(z)\psi^*(z)\!:$,
 $e_2(z)$, $h_2(z)$, $f_2(z)$,
 and
 $\psi^*(z)$.
 We have
 $A=h_2/2$,
 and
 \begin{align*}
\hat{A}(z)=h_2(z) + t^2 J^{h_1}(z)e_2(z).
\end{align*}
We find that
the isomorphism
 \begin{align*}
\mi_{t}: V^k(\g)\* L_1(\g)\* \bigwedge\nolimits^{\infty/2+\bullet}(\g_{>0})\isomap V^k(\g)\* L_1(\g)
\* \bigwedge\nolimits^{\infty/2+\bullet}(\g_{>0})
\end{align*}
is given by
\begin{align*}
&e_1(z)\mapsto e_1(z)(1-t^2 e_2(z)),
\\
&h_1(z)\mapsto h_1(z)-t^2 k\partial_z e_2(z),\\
&f_1(z)\mapsto f_1(z)(1+t^2 e_2(z)),\\
&e_2(z)\mapsto e_2(z),\\
&h_2(z)\mapsto h_2(z)+t^2J^{h_1}(z)e_2(z),\\
&f_2(z)\mapsto
f_2(z) -\frac{t^2}{2}J^{h_1}(z)h_2(z)
+ \partial_zJ ^{h_1}(z)-\frac{t^4}{4}
:\!e_2(z) J^{h_1}(z)^2\!:,
\\
&\psi(z)\mapsto  \psi(z)(1-t^2 e_2(z)),\\
&\psi^*(z)\mapsto  \psi^*(z)(1+t^2e_2(z)).
\end{align*}
Note that 
we have  $(1-t^2 e_2(z))(1+t^2 e_2(z))=1$  on $\L_1(\g)$.
\end{Ex}

\section{Remarks on superalgebras}\label{sec:super}

We restricted our attention to vertex algebras and here we remark that the results also hold for vertex superalgebras, i.e.
\begin{Rem}
All results of section \ref{section:auto} also hold if we allow $V$ to be a vertex superalgebra, such that $A$ and $\hat{A}$ are even elements of $V$.
\end{Rem}

\begin{Rem}
In the proofs of section \ref{section:proof} we only used the following properties of $\L_{\ell}(\g)$: 
\begin{enumerate}
\item finite dimensionality and boundedness from above of conformal weight spaces with respect to the twisted action
of $(\omega_{\L_{\ell}(\g)})_{(1)}$
corresponding to Li's delta operator
associated with $-x_0$; 
\item semisimple action of $A_0$, where $A:=\pi_2(x_0)\in \L_{\ell}(\g)$.
\end{enumerate}
Thus the results are also true by replacing $\L_{\ell}(\g)$ by any vertex algebra that carries an action of $V^\ell(\g)$ and satisfies these two properties.

We can also allow $\g$ to be a Lie superalgebra and $f$ an even nilpotent element, such that $\L_{\ell}(\g)$ satisfies above two properties.

\end{Rem}
Here are  examples of simple affine vertex superalgebras that satisfy the two properties of the Remark. 

Firstly if we take $\g = \mathfrak{osp}_{1|2}$ and $\ell$ a positive integer $\ell$, then $\L_{\ell}(\g)$ is rational \cite{CFK}. A similar statement is expected to be true only for $\g = \mathfrak{osp}_{1|2n}$.

Secondly, it is well-known \cite{KacWak01} that $\L_1(\mathfrak{sl}_{n|m})$ is a conformal extension of $\L_1(\mathfrak{sl}_n) \otimes \L_{-1}(\mathfrak{sl}_m) \otimes \pi$ with $\pi$ a rank one Heisenberg vertex algebra. Thus if we choose the $\mathfrak{sl}_2$-triple $\{ e, h, f\}$ for the reduction to be a subalgebra of  the $\mathfrak{sl}_n$  subalgebra of $\mathfrak{sl}_{n|m}$, 
\[
\{ e, h, f\} \subset \mathfrak{sl}_n \subset \mathfrak{sl}_{n|m},
\]
 then the two properties of the above remark are satisfied.  Moreover for a positive integer $\ell$ one has the embedding of $\L_\ell(\mathfrak{sl}_n)$ in $\L_1(\mathfrak{sl}_n)^{\otimes \ell}$. It follows that $\L_\ell(\mathfrak{sl}_n)$ embeds into the diagonal affine vertex subalgebra of $\mathfrak{sl}_{n|m}$ of level $\ell$ of $\L_1(\mathfrak{sl}_{n|m})^{\otimes \ell}$ and hence especially $\L_\ell(\mathfrak{sl}_n)$ is a subalgebra of the simple quotient $\L_\ell(\mathfrak{sl}_{n|m})$. As a consequence the two properties of the remark hold for $\L_\ell(\mathfrak{sl}_{n|m})$ if we choose the $\mathfrak{sl}_2$-triple $\{ e, h, f\}$ for the reduction to be a subalgebra of  the $\mathfrak{sl}_n$  subalgebra of $\mathfrak{sl}_{n|m}$.

In general, the notion of integrable highest-weight modules of affine Lie superalgberas is due to Kac and Wakimoto and they are necessarily integrable modules with respect to a certain affine vertex subalgebra, call it $\L_\ell(\mathfrak a)$, \cite[Thm. 6.1 and eqn. 8.2]{KacWak01}, so that if $\{ e, h, f\} \subset \mathfrak{a}$, then our conditions are also satisfied. 

\section{Urod conformal vector}
\label{section:Urod conformal vector}
In this section we assume that 
$k$ is not critical,
and discuss
how the conformal
structure match under the isomorphism
\begin{align*}
\mi:\W^k(\g,f)\* \L_{\ell}(\g)\cong H_{DS,f}^0(V^k(\g)\* \L_{\ell}(\g))
\end{align*}
of vertex algebras.

On the right-hand-side
the conformal vector 
of of $H_{DS,f}^0(V^k(\g)\* \L_{\ell}(\g))$ 
is given by
\begin{align*}
\omega_{total}=\omega_{V^k(\g)}+\omega_{\L_{\ell}(\g)}+
\omega_{F_{\chi}}+\omega_{\bigwedge\nolimits^{\infty/2+\bullet}(\g_{>0})}+T \pi_1(x_0)+T\pi_2(x_0),
\end{align*}
where
$\omega_{V^k(\g)}$ is the Sugawara conformal vector 
of $V^k(\g)$.
It has the  central charge
\begin{align}
\frac{k\dim  \g}{k+h^{\vee}}&
+\frac{\ell\dim  \g}{\ell+h^{\vee}}
-\frac{(k+\ell)\dim  \g}{k+\ell+h^{\vee}}\\
&+
\dim \g_0-\frac{1}{2}\dim \g_{1/2}-12|\frac{\rho}{\sqrt{k
+\ell+h^{\vee}}}-\sqrt{k+\ell+h^{\vee}}x_0|^2.
\nonumber
\end{align}
Clearly,
$\mi^{-1}(\omega_{total})$  is a conformal vector
of $\W^k(\g,f)\* \L_{\ell}(\g)$.
Set
\begin{align}
\omega_{Urod}=\mi^{-1}(\omega^{total})-\omega_{\W^k(\g,f)},
\end{align}
where 
$\omega_{\W^k(\g,f)}$ is the conformal vector of 
$\W^k(\g,f)$.
Since  it commutes with 
$\omega_{\W^k(\g,f)}$, 
$\omega_{Urod}$ defines a conformal vector of 
$\L_{\ell}(\g)$
(\cite[3.11]{LepLi04}),
 which is called 
{\em the Urod conformal vector} of $\L_{\ell}(\g)$.
It depends on the choice of $x_0$ and $k$,
and has the
 central charge
\begin{align}
\label{eq:cc-of-Ugot}
\frac{k\dim  \g}{k+h^{\vee}}&
+\frac{\ell\dim  \g}{\ell+h^{\vee}}
-\frac{(k+\ell)\dim  \g}{k+\ell+h^{\vee}}\\
+
12&\left(|\frac{\rho}{\sqrt{k
+h^{\vee}}}-\sqrt{k+h^{\vee}}x_0|^2
-|\frac{\rho}{\sqrt{k
+\ell+h^{\vee}}}-\sqrt{k+\ell+h^{\vee}}x_0|^2\right).
\nonumber
\end{align}
In the case that $\g$ is simply-laced and $f=f_{prin}$,
the central charge \eqref{eq:cc-of-Ugot} becomes 
\begin{align}
\label{eq:cc-of-principal-Ugot}
-\frac{\ell(\ell h+h^2-1)\dim \g}{\ell+h},
\end{align}
where $h$ is the Coxeter number
and we have used the strange formula
$|\rho|^2/2h^{\vee}=\dim \g/24$,
and so it does not depend on the parameter $k$.

By definition,
the conformal vertex algebra
$\W^k(\g,f)\* \L_{\ell}(\g)$
with the conformal vector 
$\omega_{\W^k(\g,f)}+\omega_{Urod}$
is isomorphic to 
$H_{DS,f}^0(V^k(\g)\* \L_{\ell}(\g))
$
as conformal vertex algebras.

\begin{Lem}\label{Lem:Hamiltonian-of-Urod}
The Hamiltonian of $\L_{\ell}(\g)$ defined by
$\omega_{Urod}$
coincides with 
$$H_{Urod}:=(\omega_{\L_{\ell}(\g)})_{(1)}-(x_0)_{(0)}.$$
\end{Lem}
\begin{proof}
For a homogenous element $x\in \g$,
$\pi_2(x)\in C$ has the conformal weight 
$1-\deg x$.
Hence,
$\tilde{x}$ has the the conformal weight 
$1-\deg x$
as well.
\end{proof}
\begin{Ex}(Continued from Example \ref{Ex:sl2}.)
The Urod conformal vector of $\L_1(\g)$ is given by
\begin{align*}
\omega_{Urod}=\omega_{\L_1(\g)}+T h/2-(k+1)T^2 e/2,
\end{align*}
which agrees with \cite{BerFeiLit16}.
\end{Ex}

\section{Compatibility with twists}\label{Section:twists}
In this section
we show that
 the various {\em twisted} Drinfeld-Sokolov reduction functors
 commute with tensoring integrable representations as well.

For $\cmu \in \check{P}$,
we define a character $\hat{\chi}_\cmu$ of
$\g_{\geq 1}[t,t^{-1}]$ by the formula
\begin{equation}    \label{Psimu}
\hat{\chi}_\cmu(x_\alpha f(t)) = \chi(x_{\alpha}) \cdot \on{Res}_{t=0} f(t)
t^{\bra \cmu,\alpha \ket} dt, \qquad f(t) \in \C[t,t^{-1}].
\end{equation}
Define $$F_{\chi,\cmu}=U(\g_{>0}[t,t^{-1}])\otimes_{U(\g_{>0}[t]+\g_{\geq 1}[t,t^{-1}])}\C_{\hat\chi_{\cmu}},$$
where $\C_{\hat \chi_{\cmu}}$ is the one-dimensional representation of 
$\g_{>0}[t]+\g_{\geq 1}[t,t^{-1}]$
on which $\g_{\geq 1}[t,t^{-1}]$ acts by  the character $\hat \chi_{\cmu}$
and  $\g_{>0}[t]$ acts trivially.
Then $F_{\chi,\cmu}$ is naturally a $V(\g_{>0})$-module and 
we denote by  $\Phi_{\alpha}^{\cmu}(z)$ the image of $x_{\alpha}(z)$ in 
$(\End F_{\chi,\cmu}))[z,z^{-1}]$.
We have
\begin{align*}
&\Phi_{\alpha}^{\cmu}(z)=z^{\bra \cmu,\alpha\ket}\chi(x_{\alpha}) \quad \text{for }\alpha\in \Delta_{>0},\\
&\Phi_{\alpha}^{\cmu}(z)\Phi_{\beta}^{\cmu}(w)\sim \frac{w^{\bra \cmu,\alpha+\beta\ket}
\chi([x_{\alpha},x_\beta])}{z-w.}\end{align*}

For a smooth $\affg$-module $M$ of level $k$,
we define
\begin{align}
H_{DS,f,\cmu}^{\bullet}(M)=H^{\infty/2+\bullet}(\g_{>0}[t,t^{-1}], M\* F_{\chi,\cmu}),
\end{align}
where 
$\g_{>0}[t,t^{-1}]$ acts diagonally on $M\* F_{\chi,\cmu}$.
By definition,
$H_{DS,f,\cmu}^{\bullet}(M)$ is the cohomology
 of the complex
 $(C_{\cmu}(M), (Q_{\cmu})_{(0)})$,
 where 
$C_{\cmu}(M)=M\* F_{\chi,\cmu}\*\bigwedge\nolimits^{\infty/2+\bullet}(\g_{>0})
$
and
$(Q_{\cmu})_{(0)}$ is the zero-mode of the field 
\begin{align*}
Q_{\cmu}(z)=\sum_{\alpha\in \Delta_{>0}}x_\alpha(z)\psi_{\alpha}^*(z)
+\sum_{\alpha\in \Delta_{>0}}\Phi_{\alpha}^{\cmu}(z)\psi_{\alpha}^*(z)
-\frac{1}{2}
\sum_{\alpha,\beta,\gamma\in \Delta_{>0}}c_{\alpha,\beta}^\gamma
\psi_{\alpha}^*(z)\psi_{\beta}^*(z)\psi_{\gamma}(z),
\end{align*}
on $C_{\cmu}(M)$.
 
 We define the structure of a $\W^k(\g,f)$-module on
$H_{DS,f,\check{\mu}}^i(M), i \in \Z$, as follows.
Let
$\Delta(J^{\{\cmu\}},z)$ be Li's delta operator 
 corresponding to the field
\begin{equation}    \label{cmuz}
J^{\{\cmu\}}(z):=\cmu(z) + \sum_{\al \in \Delta_{>0}} \bra \alpha,\cmu \ket :\! \psi_\al(z)
\psi^*_\al(z) \! :
-\frac{1}{2}\sum_{\al \in \Delta_{1/2}}\bra \alpha,\cmu\ket 
 :\! \Phi_\al(z)
\Phi^\al(z) \! :
\end{equation}
in $C(V^k(\g))$, where 
$\Phi^{\alpha}(z)$ is a linear sum of $\Phi^{\beta}(z)$
corresponding to the vector $x^{\alpha}$ dual to $x_{\alpha}$ 
with respect to the symplectic form 
\eqref{eq:symp-form-on},
that is,
$(f|[x_{\alpha},x^{\beta}])=\delta_{\alpha,\beta}$.
We have
\begin{align*}
J^{\{\cmu\}}(z)x_{\alpha}(w)\sim \frac{\bra \alpha,\cmu \ket }{z-w}x_\al(w),
\quad J^{\{\cmu\}}(z)\Phi_{\alpha}(w)\sim \frac{\bra \alpha,\cmu \ket }{z-w}\Phi_\al(w),\\
J^{\{\cmu\}}(z)\psi_{\alpha}(w)\sim \frac{\bra \alpha,\cmu \ket }{z-w}\psi_\al(w),
\quad J^{\{\cmu\}}(w)\psi^*_{\alpha}(z)\sim -\frac{\bra \alpha,\cmu \ket }{z-w}\psi^*_\al(w),
\end{align*}
see \cite{KacRoaWak03}.
For any smooth $\affg$-module $M$,
we can twist the action of 
$C(V^k(\g))$ on 
$C(M)$
by the correspondence
\begin{equation}    \label{DeltaLi}
\sigma_\cmu : Y(A,z)  \mapsto Y_{C(M)}(\Delta(J^{\{\cmu\}},z)A,z).
\end{equation}
for any $A\in C(V^\kappa(\g))$.
Since we have
\begin{align*}
& \sigma_\cmu(x_{\alpha}(z))=z^{- \bra \alpha,\cmu \ket }x_{\alpha}(z),\quad
\sigma_\cmu(\Phi_{\alpha}(z))=z^{- \bra \alpha,\cmu \ket }\Phi_{\alpha}(z),\\
& \sigma_\cmu(\psi_{\alpha}(z))=z^{- \bra \alpha,\cmu \ket }\psi_{\alpha}(z),
\quad
\sigma_\cmu(\psi^*_{\alpha}(z))=z^{\bra \alpha,\cmu \ket }\psi^*_{\alpha}(z),
\end{align*}
it follows that
the resulting
complex $(C(M),\sigma_\cmu(Q_{(0)}))$
is naturally identified with 
 $(C_{\cmu}(M), (Q_{\cmu})_{(0)})$.
Therefore,
 $(C_{\cmu}(M), (Q_{\cmu})_{(0)})$
 has the structure of a differential graded module over the differential vertex algebra
 $(C(V^k(\g)),Q_{(0)}))$,
 and hence,
 each cohomology $H_{DS,f,\check{\mu}}^i(M), i \in \Z$,
 is a module over $\W^k(\g,f)=H_{DS,f}^0(M)$.
 
 The functor $H_{DS,f,\check{\mu}}^0(?)$ was introduced in \cite{AraFre19} 
 in the case that $f$ is a principal nilpotent element.
 
Let 
  $V$ be vertex algebra
  equipped with a vertex algebra homomorphism
  $V^k(\g)\ra V$,
   and let $M$ be a $V$-module.
  Then
the same construction as above gives  $H_{DS,f,\check{\mu}}^i(M)$
a structure of  a module over the vertex algebra 
$H_{DS,f}^0(V)$.

\begin{Th}\label{Th:compatibility-with-twist}
Let $V$ be a quotient of the universal affine vertex algebra $V^k(\g)$,
 $\ell\in \Z_{\geq 0}$,
 and let  $\cmu \in \check{P}$.
For any $V$-module $M$,
$\L_{\ell}(\g)$-module $N$,
and $i\in \Z$,
there is an isomorphism
\begin{align*}
H_{DS,f,\cmu}^i(M\*N)\cong H_{DS,f,\cmu}^i(M)\*  \sigma_\cmu^* N
\end{align*}
as modules over 
$H_{DS,f}^0(V\otimes \L_{\ell}(\g))\cong H_{DS,f}^0(V)\otimes \L_{\ell}(\g)$,
where $\sigma_\cmu^* N$
is the twist of $N$ on which $A\in \L_{\ell}(\g)$ acts as 
$\sigma_{\cmu}(A(z))$.
\end{Th}
\begin{proof}
Since $\L_{\ell}(\g)$ is rational,
both $A_{(0)}$ and $\tilde{A}_{(0)}$
acts semisimply on 
$C_{\mu}(M\* N)$ with respect to the twisted action
of $C(V^k(\g)\* \L_{\ell}(\g))$ described above,
where $A$ is defined in  \eqref{eq:A-for-translation}.
Therefore the  assertion follows immediately from Lemma \ref{Lem:iso2}.
\end{proof}

More generally,
let $$w=(y,\cmu)$$ be an element of the extented
affine Weyl group  $\tilde{W}=W\ltimes P^{\vee}$,
where $W$ is the Weyl group of $\g$,
and let 
$\tilde{y}$ be a Tits lifting of $y$ to 
and automorphism of $\g$,
so that $\tilde{y}(x_{\alpha})=c_{\alpha}x_{y(\alpha)}$
for some $c_{\alpha}\in \C^*$.
Then
\begin{align*}
\sigma_w: x_{\alpha}t^n\mapsto \tilde{y}(x_{\alpha})t^{n-\bra \alpha,\cmu\ket}
\end{align*}
defines a Tits lifting of $w$ to an automorphism of $\affg$.
Set
 $$F_{\chi,w}=U(\tilde{w}(\g_{>0}[t,t^{-1}]))\otimes_{U(\tilde{w}(\g_{>0}[t]+\g_{\geq 1}[t,t^{-1}]))}\C_{\hat\chi_{\cmu}},$$
where $\C_{\hat \chi_{\cmu}}$ is the one-dimensional representation of 
$\tilde{w}(\g_{>0}[t]+\g_{\geq 1}[t,t^{-1}])$
on which $\tilde{w}(\g_{\geq 1}[t,t^{-1}])$ 
acts by
the character $\hat \chi_{w}: x_{\alpha}t^n\mapsto ( x_{\alpha}t^n|\tilde{w}(ft^{-1})) $
and  $ \tilde{w}(\g_{>0}[t])$ acts trivially.
Then
\begin{align}
H_{DS,f,w}^{i}(M)=H^{\infty/2+i}(y(\g_{>0})[t,t^{-1}], M\* F_{\chi,w}),
\end{align}
is equipped with a $\W^k(\g,f)$-module structure.
%
%
%

The proof of the following assertion is the same as that of 
Theorem \ref{Th:compatibility-with-twist}.

\begin{Th}\label{Th:compatibility-with-general-twist}
Let $V$ be a quotient of the universal affine vertex algebra $V^k(\g)$,
 $\ell\in \Z_{\geq 0}$,
 and let  $w\in \tilde{W}$.
For any $V$-module $M$,
$\L_{\ell}(\g)$-module $N$,
and $i\in \Z$,
there is an isomorphism
\begin{align*}
H_{DS,f,w}^i(M\*N)\cong H_{DS,f,w}^i(M)\*  \sigma_w^* N
\end{align*}
as modules over 
$H_{DS,f}^0(V\otimes \L_{\ell}(\g))\cong H_{DS,f}^0(V)\otimes \L_{\ell}(\g)$.
\end{Th}
In the above
 $\sigma_w^* N$
is the twist of $N$ on which $A\in \L_{\ell}(\g)$ acts as 
$\sigma_{w}(A(z))$,
which is again an integrable representation of $\affg$ of level $\ell$.

Suppose that the grading 
the \eqref{eq:grading} is even, that is,
$\g_j=0$ unless $j\in \Z$.
Then $x_0\in P^{\vee}$.
For a smooth
$\affg$-module $M$
of level $k$,
^^ ^^ $-$"-Drinfeld-Sokolov reduction
$H_{DS,f,-}^\bullet(M)$ (\cite{FKW92,Ara08-a})
is nothing but the twisted reduction
for $w=(w_0,-x_0)$,
where $w_0$ is the longest element of $W$:
\begin{align}
H_{DS,f,-}^\bullet(M)=H_{DS,f,(w_0,-x_0)}^{\bullet}(M)
\end{align}
\begin{Co}\label{Co:compatibility-minus-reduction}
Let $V$ be a quotient of the universal affine vertex algebra $V^k(\g)$,
 $\ell\in \Z_{\geq 0}$.
 Suppose that the grading 
the \eqref{eq:grading} is even.
For any $V$-module $M$,
$\L_{\ell}(\g)$-module $N$,
and $i\in \Z$,
there is an isomorphism
\begin{align*}
H_{DS,f,-}^i(M\*N)\cong H_{DS,f,-}^i(M)\*  \sigma_{(w_0,-x_0)}^* N
\end{align*}
as modules over 
$H_{DS,f}^0(V\otimes \L_{\ell}(\g))\cong H_{DS,f}^0(V)\otimes \L_{\ell}(\g)$.
\end{Co}

\section{Higher rank Urod algebras}\label{Sec:withGKO}

In the case that
$f$ is a principal nilpotent element $f_{prin}$,
we 
denote by $\W^k(\g)$  the principal $W$-algebra
$\W^k(\g,f_{prin})$,
and by $\W_k(\g)$ the unique simple quotient of $\W^k(\g)$.
We have the {\em Feigin-Frenkel duality} (\cite{FF91,AFO}, see also \cite{ACL19})
which states that
\begin{align}
\W^k(\g)\cong \W^{\check{k}}({}^L\g),
\label{eq:FF-duality}
\end{align}
where ${}^L\g$ is the Langlands dual Lie algebra,
$\check{k}$ is defined by the formula $r^{\vee}(k+h^{\vee})(\check{k}+{}^Lh^{\vee})=1$.
Here,
$r^{\vee}$ is the lacety of $\g$ and
${}^Lh^{\vee}$ is the dual Coxeter number of $\g$.

Suppose that  
$\g$ is simply laced and $k+h^{\vee}-1\not\in
\Q_{\leq 0}$.
By \cite{ACL19},
we have
\begin{align*}
\W^{\ell}(\g)\cong (V^{k-1}(\g)\* \L_1(\g))^{\g[t]},
\end{align*}
that is,
$\W^{\ell}(\g)$ is isomorphic to the commutant of 
$V^{k}(\g)$ in $V^{k-1}(\g)\* \L_1(\g)$,
where $V^k(\g)$ is considered as a vertex subalgebra of 
 $V^{k-1}(\g)\* \L_1(\g)$ by the diagonal embedding and
 $\ell$ is defined by \eqref{eq:def-of-ell}.
 In other words,
we have a conformal
vertex algebra embedding
\begin{align}
V^{k}(\g)\* \W^{\ell}(\g)\hookrightarrow V^{k-1}(\g)\* \L_{1}(\g).
\label{eq:GKO}
\end{align}
By Main Theorem \ref{MainTh:iso},
\eqref{eq:GKO} induces the  conformal  vertex algebra homomorphism 
\begin{align}
\W^k(\g,f)\*\W^{\ell}(\g)\ra 
H_{DS,f}^0(V^{k-1}(\g)\* \L_{1}(\g))
\overset{\mi^{-1}}{\isomap} \W^{k-1}(\g,f)\* \mc{U}(\g,f).
\label{eq:Urod-embedding}
\end{align}
which is again embedding
by Theorem \ref{Th:vanishing}.
Here, $\mc{U}(\g,f)$ is the vertex algebra
  $\L_1(\g)$  equipped with the Urod conformal vector $\omega_{Urod}$,
  which we call the {\em Urod algebra}
  associated with $(\g,f)$. 
  We set $$\mc{U}(\g)=\mc{U}(\g,f_{prin}).$$

%
%
 
 \subsection{Generic decompositions}
 In this subsection we assume that
 $k$ is irrational.
 
 For $(\lam,\cmu)\in P_+\times \check{P}_+$,
define 
\begin{align*}
T_{\lam,\cmu}^{k}=H_{DS,f_{prin},\cmu}^0(\We{k}{\lam})\in \W^{k}(\g)\on{-Mod}.
\end{align*}
It was shown in \cite{AraFre19} that the  $\W^k(\g)$-modules
$T_{\lam,\cmu}^k$ are simple, and
the isomorphism
\begin{align}
\label{eq:langland-duality}
T_{\lam,\cmu}^k\cong T_{\cmu,\lam}^{\check{k}}
\end{align}
holds
under the Feigin-Frenkel duality  \eqref{eq:FF-duality}.

 Let $\ell$ be non-negative integer.
Then $\{\Si{\ell}{\lam}\mid \lam \in P^\ell_+\}$
gives  the complete set of isomorphism classes of simple
$\L_{\ell}(\g)$-modules,
where $$ P^\ell_+=\{\lam\in P_+\mid \bra \lam,\theta^{\vee}\ket \leq \ell\}\subset P_+$$ is the set of integrable dominant weight of $\g$ of level $\ell$.
In particular,
 $\{\Si{1}{\lam}\mid \lam \in P^1_+\}$
 gives  the complete set of isomorphism classes
 of simple $\mc{U}(\g,f)$-modules.

Let 
\begin{align}
\We{k}{\mu,f}=H_{DS,f}^0(\We{k}{\mu})\in \W^k(\g,f)\on{-Mod}.
\end{align}

\begin{Th}\label{Th:generic-decom}
Let  $\g$ be simply-laced and
 let
$f$ be any nilpotent element of $\g$.
 For $\lam,\mu\in P_+$ and $\nu\in P_+^1$,
we have
\begin{align*}
\We{k-1}{\mu,f}\* \mathcal{U}(\g,f)\cong \bigoplus_{\substack{\lam\in P_+\\ \lam-\mu-\nu\in Q}}\We{k}{\lam,f}\*  T_{\lam, \mu}^{\ell}
\end{align*}
as  $\W^{k}(\g,f)\*\W^{\ell}(\g)$-modules (see \eqref{eq:Urod-embedding}).
\end{Th}

In the case $f=f_{prin}$ we have the following more general statement.
\begin{Th}\label{Th:dec:generic-principal}
Let  $\g$ is simply-laced.
For $\lam,\mu,\mu'\in P_+$ and $\nu\in P_+^1$,
we have
\begin{align}
T^{k-1}_{\mu,\mu'}\*  \mc{U}(\g)
=\bigoplus_{\substack{\lam\in P_+\\ \lam-\mu-\mu'-\nu\in Q}}T^k_{\lam,\mu'}\*  T_{\lam, \mu}^{\ell}
\label{eq:twiste-dec}
\end{align}
as  $\W^{k}(\g)\*\W^{\ell}(\g)$-modules.
\end{Th}
Note that  Theorem \ref{Th:dec:generic-principal} is compatible with  \eqref{eq:langland-duality}
 since $\check{(k-1)}=\ell-1$.

\begin{proof}[Proof of Theorem \ref{Th:generic-decom} and Theorem \ref{Th:dec:generic-principal}]
Let $\lam,\mu\in P_+$ and $\nu\in P_+^1$,
By \cite[Main Theorem 3]{ACL19},
 we have
\begin{align}
\We{k-1}{\mu}\* \Si{1}{\nu}=\bigoplus_{\substack{\lam\in P_+\\ \lam-\mu-\nu\in Q}} \We{k}{\lam}\*  T^{\ell}_{\lam,\mu}
\end{align}
as  $V^{k+1}(\g)\*\W^{\ell}(\g)$-modules.
Applying Main Theorem \ref{MainTh:iso} to $\We{k-1}{\mu}\* \Si{1}{\nu}$
we obtain  Theorem \ref{Th:generic-decom}. 
Next,  under the identification $P/Q\cong P^1_+$, we have 
 $\sigma_{\mu'}^*\Si{1}{\nu+\mu'+Q}\cong \Si{1}{\nu+Q}$.
 Hence we obtain 
 Theorem \ref{Th:dec:generic-principal}
by applying Theorem \ref{Th:compatibility-with-twist}.
\end{proof}

Let $\pi^k$ be the Heisenberg vertex subalgebra of $V^k(\g)$ generated by $h(z)$, $h\in \h$,
and let $\pi^k_{\lam}$ be the irreducible highest weight representation of $\pi^k$ with highest weight $\lam$.
There is a vertex algebra embedding
\begin{align*}
\W^k(\g)\hookrightarrow \pi^{k+h^\vee}
\end{align*}
called the {\em Miura map} (\cite{FeuFre90}, see also \cite{AraLec}),
and hence each $\pi^{k+h^{\vee}}_{\lam}$ is a $\W^k(\g)$-module.

\begin{Th}
\label{Th:dec-of-free-field}
Let  $\g$ is simply-laced,
 $\mu\in \h^*$ be generic, $\nu\in P^1_+$.
We have the isomorphism
\begin{align*}
\pi_{\mu}^{k-1+h^{\vee}}\* \Si{1}{\nu}\cong \bigoplus_{\lam \in \h^*
\atop \lam-\mu-\nu\in Q} \pi^{k+h^{\vee}}_{\lam}\* \pi^{\ell+h^{\vee}}_{\lam-(\ell+h^{\vee})\mu}
\end{align*}
as $\W^k(\g)\* \W^\ell(\g)$-modules.
\end{Th}

For a generic $\lam\in \h^*$,
the $\W^k(\g)$-module $\pi^{k+h^{\vee}}_{\lam}$ is irreducible and  isomorphic to a Verma module 
$\mathbb{M}^k(\chi_{\lam})$ with highest weight
(\cite{FF90,Fre92,Ara07}).
  Here 
 $\chi_{\lam}:\on{Zhu}(\W^k(\g))\ra \C$ is described in \cite[(27)]{ACL19},
 where $\on{Zhu}(\W^k(\g))$ is the Zhu algebra of $\W^k(\g)$.
 Hence the following assertion immediately follows from Theorem \ref{Th:dec-of-free-field}.

\begin{Co}
\label{Th:dec-of-Verma}
Let  $\g$ is simply-laced,
 $\mu\in \h^*$ be generic, $\nu\in P^1_+$.
We have the isomorphism
\begin{align*}
\mathbb{M}^{k-1}(\chi_{\mu})\* \Si{1}{\nu}\cong \bigoplus_{\lam \in \h^*
\atop \lam-\mu-\nu\in Q} \mathbb{M}^k(\chi_{\lam})\*\mathbb{M}^{\ell}(\chi_{\lam-(\ell+h^\vee)\mu})
\end{align*}
as $\W^k(\g)\* \W^\ell(\g)$-modules.
\end{Co}

 Let $\mathbb{W}^k_{{\lam}}$ be the Wakimoto module \cite{FeuFre90}
 of $\affg$
 at level $k$ with highest weight $\lam\in \h^*$.
For $\lam=0$,
$\mathbb{W}^k:=\mathbb{W}^k_{0}$ is a vertex algebra,
and we have an embedding
$V^k(\g)\hookrightarrow \mathbb{W}^k$ of vertex algebras (\cite{FeuFre90, Fre05}).
We have
$H_{DS,f_{prin}}^i( \mathbb{W}^k_{\lam})\cong  \delta_{i,0}\pi^{k+h^{\vee}}_{\lam}$,
and the Miura map 
is by definition \cite{FeuFre90} the
map
$\W^k(\g)\ra H_{DS,f_{prin}}^0( \mathbb{W}^k)\cong \pi^{k+h^{\vee}}$
induced by the embedding $V^k(\g)\hookrightarrow \mathbb{W}^k$.

\begin{proof}[Proof of Theorem \ref{Th:dec-of-free-field}]
Since $\nu$ is generic,
$\mathbb{W}^{k-1}_{\mu}\* \Si{1}{\nu}$
is a direct sum of irreducible Verma modules
as a diagonal $\affg$-module,
or equivalently, a direct sum of irreducible Wakimoto modules
$\mathbb{W}^{k}_\lam$.
So we can write
$\mathbb{W}^{k-1}_{\mu}\* \Si{1}{\nu}=\bigoplus_{\lam\in \h^*}\mathbb{W}^{k}_\lam\* m_{\mu,\nu}^\lam$,
where
 $m_{\mu,\nu}^\lam$ be the multiplicity of 
 $\mathbb{W}^{k}_\lam$ in $\mathbb{W}^{k-1}_{\mu}\* \Si{1}{\nu}$.
 Note that 
 $m_{\mu,\nu}^\lam$ 
 is a  $\W^{\ell}(\g)$-module 
by \eqref{eq:GKO}.
We have
\begin{align*}
m_{\mu,\nu}^\lam\cong \Hom_{\pi^{k+1+h^\vee}}(\pi^{k+1+h^{\vee}}_{\lam},
H^{\infty/2+0}(\n[t,t^{-1}],\mathbb{W}^{k-1}_{\mu}\* \Si{1}{\nu})\\
\cong 
\begin{cases}
\pi^{\ell+h^{\vee}}_{\mu-(\ell+h^{\vee})\lam}&\text{if }\lam-\mu-\nu\in Q,\\
0&\text{otherwise, }
\end{cases}
\end{align*}
as $\W^{\ell}(\g)$-modules
by \cite[Proposition 8.3]{ACL19}.
The assertion follows by
 applying Main Theorem \ref{MainTh:iso}
to 
\begin{align}
\mathbb{W}^{k-1}_{\mu}\* \Si{1}{\nu}=\bigoplus\limits_{\lam\in \h^*
\atop \lam-\mu-\nu\in Q}\mathbb{W}^{k}_\lam\* \pi_{\mu-(\ell+h^{\vee})}.
\end{align}
\end{proof}

\subsection{Decomposition at admissible levels}
\label{subsection:Decomposition at admissible levels}
Let  $k$ be an admissible number for $\widehat{\g}$,
that is,
 $\L_k(\g)$ is admissible (\cite{KacWak89}) as a $\widehat{\g}$-module.
 In the case that $\g$ is simply-laced,
 this condition is equivalent to 
\begin{align}
\label{eq:admissibleno}
k+h^{\vee}=\frac{p}{q},\quad p,q\in \Z_{\geq 1}, \ (p,q)=1,\  p\geq h^{\vee}.
\end{align}

For an admissible number $k$,
a simple module over $\L_k(\g)$ 
need not be ordinary, that is, in $\on{KL}$, unless $k$ is a non-negative integer.
The classification of simple highest weight representations of
 $\L_k(\g)$ was given in \cite{A12-2}.
For our purpose, we need only the 
ordinary representations of $\L_k(\g)$.
By \cite{A12-2},
the complete set of isomorphism classes 
of ordinary simple $\L_k(\g)$-modules 
is given by
\begin{align*}
\{\Si{k}{\lam}\mid \lam \in Adm_{\Z}^k\},
\end{align*}
where $Adm_{\Z}^k=\{\lam\in P_+\mid \L^k_{\lam}\text{ is admissible}\}$.
We have
\begin{align}
Adm_{\Z}^k=P^{p-h^{\vee}}_+,
\end{align}
if 
$\g$ is simply-laced and $k$ is of the form  \eqref{eq:admissibleno}.

Let $k$ be an admissible number
and $\lam\in Adm^k_{\Z}$.
By \cite{Ara09b},
the associated variety $X_{\Si{k}{\lam}}$ \cite{Ara12}
of $\Si{k}{\lam}$ is
the closure
of some
 nilpotent orbit $\mathbb{O}_k$
which depends only on the denominator
$q$ of $k$.
More explicitly,
we have 
\begin{align}
X_{\Si{k}{\lam}}
=\{x\in \g\mid (\ad x)^{2q}=0\}
\end{align}
in the case that $\g$ is simply-laced.
By \cite{Ara09b},
we
 have
\begin{align}
H_{DS,f}^0(\Si{k}{\lam})\ne 0
\quad \iff \quad f \in  X_{\Si{k}{\lam}}=\overline{\mathbb{O}}_k.
\end{align}

An admissible number $k$ is called {\em non-degenerate}
if $ X_{\Si{k}{\lam}}$ equals to the nilpotent cone $\mc{N}$
of $\g$
for some
$\lam\in Adm^k_{\Z}$,
or equivalently, for all $Adm^k_{\Z}$.
In the case that $\g$ is simply-laced,
this happens if and only if the denominator $q$ of $k$ is equal or greater than $h^{\vee}$.
In this is the case,
the simple principal $W$-algebra
$\W_k(\g)=\W_k(\g,f_{prin})$ is rational and lisse (\cite{Ara09b,A2012Dec}),
and the complete set of the isomorphism classes of 
$\W_k(\g)$-module
is given by
\begin{align}
\{\mathbf{L}^k_{[\lam,\cmu]}\mid [\lam,\cmu]\in (Adm_{\Z}^k\times Adm_{\Z}^{\check{k}})/\tilde{W}_+\},
\label{eq:simple-modules-of-minimal-model}
\end{align}
where 
\begin{align}
\mathbf{L}^k_{[\lam,\cmu]}:=H_{DS,f_{prin}}^0(\Si{k}{\lam-(k+h^{\vee})\cmu})
\end{align}
and $\tilde{W}_+$ is the subgroup of the extended affine Weyl group of $\g$
consisting of elements of length zero that acts on the set 
$Adm_{\Z}^k\times Adm_{\Z}^{\check{k}}$ diagonally.
We have 
\begin{align}
\mathbf{L}^k_{[\lam,\cmu]}\cong \mathbf{L}^{\check{k}}_{[\cmu,\lam]}
\end{align}
under the Feigin-Frenkel duality.

The following assertion is new for nonzero $\cmu$.
\begin{Th}
Let $k$ be a non-degenerate admissible number,
and let $\lam\in Adm_{\Z}^k$,
$\cmu\in Adm_{\Z}^{\check{k}}$.
We have
$$H_{DS,f_{prin},\cmu}^i(\Si{k}{\lam})\cong \begin{cases}
\mathbf{L}^k_{[\lam,\cmu]}&\text{for }i=0,\\
0&\text{for }i\ne 0
\end{cases}$$
as $\W^k(\g)$-modules.
\end{Th}
\begin{proof}
The case
$\cmu=0$ has been proved in \cite{Ara04, Ara07}.
In particular,
\begin{align}
H_{DS,f_{prin}}^0(\L_k(\g))\cong \W_k(\g)
\label{eq:the-image-of-simple-is-simple}.
\end{align}

By \cite[Theorem 2.1]{AraFre19},
we have $H_{DS,f_{prin},\cmu}^{i}(\We{k}{\lam})=0$ for all $i\ne 0$, $\lam\in P_+$.
It follows that $H_{DS,f_{prin},\cmu}^{i\ne 0}(M)=0$, $i\ne 0$, 
for any object $M$ in $\on{KL}$ that admits a Weyl flag.
This implies that
\begin{align}
H_{DS,f_{prin},\cmu}^{i}(M)=0\quad \text{for  }i>0,\ M\in \on{KL},
\label{eq:vanishing-ot-twisted-cohomology}
\end{align}
see the proof of \cite[Theorem 8.3]{Ara04}.

By \cite{A-BGG},
the admissible representation $\Si{k}{\lam}$
admits a two-sided resolution
\begin{align}
\label{eq:BGG}
\dots C^{-1}\ra C^0\ra C^1\ra \dots
\end{align}
of the form $C^i=\bigoplus\limits_{w\in \widehat{W}(\lam+k\Lam_0)\atop
\ell^{\infty/2}(w)=i}\mathbb{W}^k_{w\circ {\lam}}$,
where 
$\widehat{W}(\lam+k\Lam_0)$ is the integral Weyl group of $\lam+k\Lam_0$
and $\ell^{\infty/2}(w)$ is the semi-infinite length of $w$.
By \cite[Lemma 3.2]{AraFre19},
we have
$$H_{DS,f_{prin},\cmu}^0(\mathbb{W}^k_{{\lam}})\cong \begin{cases}
\pi_{\lam-(k+h^{\vee})\cmu}^k&\text{for }i=0,\\
0&\text{for }i\ne 0
\end{cases}$$
as $\W^k(\g)$-modules.
It follows  that $H_{DS,f_{prin},\cmu}^i(\Si{k}{\lam})$ is
 the $i$-th cohomology of the complex obtained by applying the functor 
 $H_{DS,f_{prin},\cmu}^0(?)$ to the resolution  \eqref{eq:BGG}.
In particular,
$H_{DS,f_{prin},\cmu}^i(\Si{k}{\lam})$ is a subquotient
of  the $\W^k(\g)$-module
$$H_{DS,f_{prin},\cmu}^0(C^i)\cong \bigoplus\limits_{w\in \widehat{W}(\lam+k\Lam_0)\atop
\ell^{\infty/2}(w)=i}\pi^k_{w\circ \lam-(k+h^{\vee})\cmu}.$$
On the other hand,
since
$\Si{k}{\lam}$ is a $\L_k(\g)$-module,
each $H_{DS,f_{prin},\cmu}^i(\Si{k}{\lam})$ is a module over the
simple $W$-algebra  $\W_k(\g)=H_{DS,f_{prin}}^0(\L_k(\g))$.
Since $\W_k(\g)$ is rational,
$H_{DS,f_{prin},\cmu}^i(\Si{k}{\lam})$
is  a direct sum of simple modules of the form 
\eqref{eq:simple-modules-of-minimal-model}.
However,
such a module
appears 
in the local composition factor
of  $\pi^k_{w\circ \lam-(k+h^{\vee})\cmu}$
if and only if $w\in W$
(\cite{Ara07}),
where $W\subset \widehat{W}(k\Lam_0)$ is the Weyl group of $\g$.
As 
\begin{align*}
\text{$\ell^{\infty/2}(w)\geq 0$ for  $w\in W$ and the equality holds if and only if $w=1$,}
\end{align*}
$H_{DS,f_{prin},\cmu}^i(\Si{k}{\lam})$ must vanish for  $i< 0$.
Together with \eqref{eq:vanishing-ot-twisted-cohomology},
we get 
$H_{DS,f_{prin},\cmu}^i(\Si{k}{\lam})=0$  for $i\ne 0$.
Finally, since
$\mathbf{L}^k_{[\lam,\cmu]}$ is the unique simple $\W_k(\g)$-module 
that appears in the local composition factor
of  $\pi^k_{\lam-(k+h^{\vee})\cmu}$
and it appears with multiplicity one (\cite{Ara07}),
$H_{DS,f_{prin},\cmu}^i(\Si{k}{\lam})$ is either zero or isomorphic to $\mathbf{L}^k_{[\lam,\cmu]}$.
The assertion follows since
 the Euler character  
of  $H_{DS,f_{prin}}^\bullet(\L_k(\g))$ is equal to
the character of $\mathbf{L}^k_{[\lam,\cmu]}$.
\end{proof}


For an admissible number $k$ and $\lam\in Adm^k_{\Z}$,
define
\begin{align}
\Si{k}{\lam,f}=H_{DS,f}^0(\Si{k}{\lam})\in \W^k(\g,f)\on{-Mod}.
\end{align}
Note that $\Si{k}{\lam,f_{prin}}=\mathbf{L}_{[\lam,0]}^k\cong \mathbf{L}_{[0,\lam]}^{\check{k}}$.

Let $\g$ be simply-laced.
Observe that
if $k-1$ is an admissible number,
then so is $k$,
and $\ell$ is a non-degenerate admissible number.
Moreover, we have
\begin{align*}
Adm_{\Z}^{k-1}=Adm_{\Z}^{\check{\ell}}
=P^{p-h^\vee}_+,\quad
Adm_{\Z}^k=Adm_{\Z}^{\ell}=P^{p+q-h^{\vee}}_+
\end{align*}
if $k-h^{\vee}+1=p/q$ with $p\geq h^{\vee}$, $q\geq 1$, $(p,q)=1$.

\begin{Th}\label{Th:decom-adm}
Let $\g$ be simply-laced,
and let  $k-1
$  
be admissible.
Suppose that $f\in X_{\L_{k}(\g)}=X_{\L_{k-1}(\g)}$.
For $\mu\in Adm^{k-1}_{\Z}$,
$\nu\in P_+^1$,
we have
\begin{align*}
\Si{k-1}{\mu,f}\otimes \Si{1}{\nu}\cong \bigoplus_{\lam\in Adm^{k}_{\Z}\atop \lam-\mu-\nu\in Q}
\Si{k}{\lam,f}\otimes  \mathbf{L}_{[\lam,\mu]}^{\ell}
\end{align*}
as $\W^{k}(\g,f)\* \W^{\ell}(\g)$-modules.
\end{Th}

In the case $f=f_{prin}$ we have the following more general statement.
\begin{Th}\label{Th:dec-admi-pri}
Let $\g$ be simply-laced,
and
let $k-1
$  be non-degenerate admissible.
For  $\mu\in Adm^{k-1}_{\Z}$,
$\mu'\in Adm^{\check{k-1}}_{\Z}=Adm^{\check{k}}_{\Z}$,
$\nu\in P_+^1$,
we have
\begin{align*}
\mathbf{L}_{[\mu,\mu']}^{k-1}\otimes  \Si{1}{\nu}\cong \bigoplus_{\lam\in Adm^k_{\Z}\atop \lam-\mu-\mu'-\nu\in Q}
\mathbf{L}_{[\lam,\mu']}^{k}\otimes  \mathbf{L}_{[\lam,\mu]}^{\ell}
\end{align*}
as $\W_{k}(\g)\* \W_{\ell}(\g)$-modules.
\end{Th}

\begin{Co}\label{Co:Urod-decomposition}
Let $\g$ be simply-laced.
Let $k+h^{\vee}=(2h^{\vee}+1)/h^{\vee}$,
so that $\ell+h^{\vee}=(2h^{\vee}+1)/(h^{\vee}+1)$.
\begin{enumerate}
\item There is an embedding of  vertex algebras
\begin{align*}
\W_k(\g)\* \W_{\ell}(\g)\hookrightarrow \L_1(\g),
\end{align*}
and $\W_k(\g)$ and $ \W_{\ell}(\g)$ form a dual pair in $ \L_1(\g)$.
\item For $\nu\in P^1_+$ we have
\begin{align*}
 \Si{1}{\nu}\cong \bigoplus_{\lam\in Adm_{\Z}^k\cap (\nu+Q)}
\mathbf{L}_{[\lam,0]}^{k}\otimes  \mathbf{L}_{[\lam,0]}^{\ell}
\end{align*}
as $\W_k(\g)\* \W_{\ell}(\g)$-modules.
\end{enumerate}
\end{Co}
\begin{proof}
The assertion follows from 
Theorem \ref{Th:dec-admi-pri}
 noting that 
$\W_{k-1}(\g)=\mathbf{L}_{[0,0]}^{k-1}=\C$ if $k+h^{\vee}-1=(h^{\vee}+1)/h^{\vee}$
or $h^{\vee}/(h^{\vee}+1)$.
\end{proof}

\begin{proof}[Proof of Theorem \ref{Th:decom-adm} and Theorem \ref{Th:dec-admi-pri}]
By \cite[Main Theorem 3]{ACL19},
we have
\begin{align*}
\Si{k-1}{\mu}\otimes \Si{1}{\nu}\cong \bigoplus_{\lam\in Adm^k_{\Z}\atop \lam-\mu-\nu\in Q}
\Si{k}{\lam}\otimes  \mathbf{L}_{[\lam,\mu]}^\ell
\end{align*}
as $\L_{k}(\g)\* \W_{\ell}(\g)$-modules.
Hence the assertion follows from Main Theorem \ref{MainTh:iso}
and Theorem \ref{Th:compatibility-with-twist}.
\end{proof}

\section{Application to the extension problem of vertex algebras}
In this section we apply Theorem \ref{Th:dec:generic-principal}
to prove the existence of the extensions of vertex algebras
that are expected by 
four-dimensional supersymmetric gauge theories (\cite{CG, CGL}).

For $\lam\in P_+$ let $\lam^*=-w_0(\lam)$,
so that $E_{\lam}^*\cong E_{\lam^*}$.
\begin{Th}\label{Th:extension}
Let $\g$ be simply-laced.
and
let $k, k'$ be irrational complex numbers satisfying
 \[
 \frac{1}{k+h^\vee} + \frac{1}{k'+ h^\vee} = n.
 \]
 for $n\in \Z_{\geq 1}$.
Then
\begin{align*}
{A}^n[\g] :=\bigoplus_{\lam\in P_+\cap Q}T^{k}_{\lam,0}\* T^{k'}_{\lam^*,0}
\end{align*}
can be equipped with a structure of simple vertex operator algebra
of central charge  
\begin{align*}
2\on{rk}\g+4h \dim \g- n h \dim \g  \left(1 + \frac{\psi^2}{n \psi-1} \right)
\end{align*}
where $\psi = k+h^\vee$ and $h$ is the Coxeter number (which equals to $h^{\vee}$).
The vertex operator algebra 
$A^n[\g]$
is of  CFT type if $n\geq 2$.
\end{Th}
\begin{proof}
We shall prove the assertion on induction on $n$ using the fact that
\begin{align}
T^{k'}_{\mu,\mu'}\*\mc{U}(\g)\cong 
\bigoplus_{\lam\in P_+\atop \lam-\mu-\mu'\in Q}
T^{k'+1}_{\lam,\mu'}\* T^{\ell}_{\lam,\mu}
\cong \bigoplus_{\lam\in P_+\atop \lam-\mu-\mu'\in Q}
T^{k'+1}_{\lam,\mu'}\* T^{\check{\ell}}_{\mu,\lam},
\label{eq:for-ext}
\end{align}
with $\check{\ell}$ satisfying the relation
 \[
 \frac{1}{\check{\ell}+ h^\vee} =   \frac{1}{k'+ h^\vee}+1
 \]
which follows from Theorem \ref{Th:dec:generic-principal} and \eqref{eq:langland-duality}.

Let us show  the assertion  for  $n=1$.
Let $\mathbf{I}_G^k$ be the  {\em chiral universal centralizer}
on $G$ at level $k$  (\cite{Arakawa:2018egx}),
which was introduced earlier in \cite{FreSty06} for the 
 $\g=\mf{sl}_2$ case
as the {\em modified regular representation of the Virasoro algebra}.
By definition, 
$\mathbf{I}_G^k$ is obtained by taking the principal Drinfeld-Sokolov reduction with respect to two commuting actions of $\affg$ on the algebra 
of the chiral differential operators \cite{MalSchVai99,BeiDri04} $\mc{D}_{G,k}^{ch}$ on $G$ at level $k$.
The $\mathbf{I}_G^k$  is a conformal vertex algebra of central charge $2 \on{rk}\g + 48(\rho|\rho^\vee)
=2 \on{rk}\g+4 h^{\vee}\dim \g$,
equipped with a conformal vertex algebra homomorphism
$\W^k(\g)\* \W^{k^*}(\g)\ra \mathbf{I}_G^k$,
where  
 $k^*$  the dual level of $k$  defined by the equation
$$\frac{1}{k+h^{\vee}}+\frac{1}{k^*+h^{\vee}}=0.$$
As explained in   \cite{Arakawa:2018egx},
$\mathbf{I}_G^k$
  is a strict chiral quantization (\cite{Arakawa:2018egx}) of the {\em universal centralizer}
$\mc{S}\times_{\g^*}(G\times \mc{S})$,
where $ \mc{S}$ is the 
Kostant-Slodowy slice.
Since $\mc{S}\times_{\g^*}(G\times \mc{S})$ is a smooth symplectic variety,
 $\mathbf{I}_G^k$ is simple (\cite{AraMorCore}).
For an irrational $k$,
we have the decomposition
$$\mathbf{I}_G^k\cong \bigoplus_{\lam\in P_+} T_{\lam,0}^k\* T_{\lam^*,0}^{k^*}$$
as  $\W^k(\g)\* \W^{k^*}(\g)$-modules,
which follows from the decomposition \cite{ArkGai02,Zhu08} of  $\mc{D}_{G,k}^{ch}$
as $\affg\times \affg$-modules.
Hence, by \eqref{eq:for-ext},
\begin{align*}
 \mathbf{I}_G^k\* \mc{U}(\g)
 \cong \bigoplus_{\lam\in P_+}T_{\lam,0}^k\* T_{\lam^*,0}^{k^*}\* \mc{U}(\g)
 \cong \bigoplus_{\lam\in P_+, \ \mu\in P_+
 \atop \mu-\lam^*\in Q}T_{\lam,0}^k \* T_{\mu,0}^{k^*+1}\* T_{\lam^*,\mu}^{\check{\ell}},
\end{align*}
and $\check{\ell}$ satisfies the relation
$$\frac{1}{k+h^{\vee}}+\frac{1}{\check{\ell}+h^{\vee}}=1.$$
It follows that 
\begin{align*}
{A}^1[\g]:=\on{Com}(\W^{k^*+1}(\g), \mathbf{I}_G^k\*\mc{U}(\g))
\cong \bigoplus_{\lam\in P_+\cap Q}T^{k}_{\lam,0}\* T^{\check{\ell}}_{\lam^*,0}.
\end{align*}
Moreover,
since $\mathbf{I}_G^k\* \mc{U}(\g)$ is simple
${A}^1[\g]$ is simple as well by  \cite[Proposition 5.4]{CreGenNak}.

Assuming that the statement is true for $n\in \Z_{\geq 1}$,
we find similarly that
$${A}^{n+1}[\g]:=\on{Com}(\W^{k'+1}(\g), {A}^{n}[\g]\* \mc{U}(\g))$$
has the required decomposition.

For the central charge computation it is useful to introduce 
\[
\psi:= k+ h^\vee\qquad \text{and} \qquad  \phi_n := \frac{\psi}{n\psi -1 },  
\]
so that 
\[
\frac{1}{\psi} + \frac{1}{\phi_n} = n.
\]
By \eqref{eq:cc-of-principal-Ugot}
and the fact that the central charge of $\W^{k+1}(\g)$
is 
\[
(1-h(h+1)(k+h)^2/(k+h+1))\on{rk}\g  = \on{rk}\g - \dim \g h   \frac{\psi^2}{\psi +1}
\]
(here we used that $ (h+1) \on{rk} \g = \dim \g$)
we have
\begin{equation*}
\begin{split}
c_{A^{n+1}[\g]} &= c_{A^{n}[\g]}-\frac{h^2+h-1}{h+1}\dim \g - \on{rk}\g   + h\frac{\psi^2}{(n\psi-1)((n+1)\psi-1)})\dim\g \\
&= - \dim \g h   + h  \phi_n \phi_{n+1}\dim\g, 
\end{split}
\end{equation*}
where 
 $c_V$ is the central charge of $V$
 and
we have put $A^{0}[\g]=\mathbf{I}_G^k$.  
Note that
\[
\phi_n - \phi_{n+1} = \phi_n\phi_{n+1} = 
\psi ( n\phi_n - (n+1) \phi_{n+1} ) 
\]
and so by induction for $n$
we have
\begin{align*}
c_{A^{n}[\g]}&=2\on{rk}\g+4h \dim \g- n h \dim \g  (1 + \psi \phi_n).
\end{align*}

The conformal dimension 
of $T_{\lam,0}^k\* T_{\lam^*,0}^{k'}$
is
\begin{align*}
\frac{n}{2}(|\lam+\rho|^2-|\rho|^2)-2(\lam|\rho)
=\frac{n}{2}|\lam|^2+(n-2)(\lam|\rho),
\end{align*}
which is an integer for $\lam \in Q$.
If $n\geq 2$,
this is clearly non-negative 
and is equal to zero if and only if $\lam=0$.
Whence the last assertion.
\end{proof}

\begin{Rem}
More generally,
it is expected \cite{CG, CGL} that 
if $k$ and $k'$ are irrational numbers related by
\[
 \frac{1}{k+h^\vee} + \frac{1}{k'+ h^\vee} = n\in \Z,
 \]
 then
\[
 \bigoplus_{\lambda \in Q \cap P^+ } \We{k}{\lambda, f} \otimes \We{k'}{\lambda^*, f'}
 \]
  can be given the structure of a simple \voa{} 
  for any nilpotent elements $f$, $f'$.
\end{Rem}

 \section{Fusion categories of modules over quasi-lisse W-algebras
}
For a vertex operator algebra $V$,
 let $\mc{C}_V^{ord}$ be the full subcategory of 
 the 
 category of finitely generated $V$-modules
 consisting of modules $M$
 on which $L_0$ acts  locally finitely,
 the $L_0$-eivenvalues of $M$ are bounded from below,
 and all the generalized $L_0$-eigenspaces are finite-dimensional.
 A simple object in $\mc{C}_V^{ord}$ is called an {\em ordinary representation} of $V$.

 Recall that a finitely strongly generated vertex algebra $V$ is called {\em quasi-lisse} \cite{AraKaw18}
 if the associate variety $X_V$ has finitely many symplectic leaves.
  By \cite[Theorem 4.1]{AraKaw18},
  if $V$ is quasi-lisse then
$\mc{C}_V^{ord}$ 
has only finitely many simple objects. 

\begin{Conj}\label{Conj:verte-tensor}
 Let $V$ a finitely strongly generated, 
 self-dual, quasi-lisse vertex operator algebra of CFT type.
Then the category
  $\mc{C}_V^{ord}$ has the structure of a vertex tensor category
 in the sense of \cite{HuaLep94}.
 \end{Conj}

In the case that $V$ is {\em lisse},
that is, $X_V$ is zero-dimensional,
Conjecture \ref{Conj:verte-tensor}
has been proved in  \cite{Huang09}. 
Conjecture \ref{Conj:verte-tensor} is true  if one can show that every object in $\mc{C}_V^{ord}$ is $C_1$-cofinite and if grading-restricted generalized Verma modules for V are of finite length \cite{CY20}.
If $V$ is a vertex algebra that has an affine vertex subalgebra at admissible level, then 
another possibility of proving Conjecture \ref{Conj:verte-tensor} is the following:
First show that the affine vertex subalgebra is simple, secondly prove that a suitable category of modules of the coset by the affine subalgebra in $V$ has vertex tensor category structure. Then use the theory of vertex algebra extensions \cite{CKM} to deduce vertex tensor category structure on $\mc{C}_V^{ord}$. This is a promising direction as for example many simple quotients of cosets of $\mathcal W$-algebras of type $A$ are rational and lisse by \cite[Cor. 6.13]{CL20} and similar results for $\mathcal W$-algebras of type $B, C$ and $D$ are work in progress.

In the case that Conjecture \ref{Conj:verte-tensor}
is true, 
we denote by $M\boxtimes_V N$ the tensor product of $V$-modules $M$ and $N$ in $\mc{C}_V^{ord}$,
or simply by $M\boxtimes N$ if no confusion should occur.
 
 \smallskip

Now let 
 $k$ be an admissible number for $\affg$.
As explained in Subsection \ref{subsection:Decomposition at admissible levels},
 $$X_{\L_k(\g)}=\overline{\mathbb{O}_k}$$
 for some 
 nilpotent orbit $\mathbb{O}_k$,
 and hence,
 $\L_k(\g)$ is quasi-lisse.
 By the conjecture of Adamovi\'{c} and Milas 
 \cite{AdaMil95} that was proved in
 \cite{A12-2},
the category $C_{\L_k(\g)}^{ord}$ is semisimple
and 
 $\{ \Si{k}{\lambda} | \lambda \in Adm^k_{\Z}\}$
 gives a complete set of isomorphism classes in  $\mc{C}_{\L_k(\g)}^{ord}$.
Note that
in the case $k$ is non-negative integer,
 this is a well-known fact \cite{FreZhu92},
 and if this is the case 
$\L_k(\g)$ is rational and lisse,
and  hence
Conjecture \ref{Conj:verte-tensor} holds by  \cite{HuaLep95}.
In the case $k$ is admissible but not an integer,
 $\L_k(\g)$  is not rational nor lisse anymore.
 Nevertheless,
Conjecture \ref{Conj:verte-tensor} has been proved 
for $V=\L_k(\g)$ in 
  \cite{CreHuaYan18}
provided that $\g$ is simply-laced.
Moreover,
it was shown in  \cite{Cre19}
that  $\mc{C}_{\L_k(\g)}^{ord}$  
is a fusion category,
i,e.,
any object is rigid,
and
we have
 an isomorphism
 \begin{align}
 \label{eq:fusion-admissible}
K[C_{\L_{p-h^{\vee}}(\g)}^{ord}]\cong K[C_{\L_k(\g)}^{ord}],
\quad [\Si{p-h^{\vee}}{\lam}]\mapsto [\Si{k}{\lam}],
\end{align}
of Grothendieck rings of our fusion categories,
where $p$ is the numerator of $k+h^{\vee}$.

%

Let $f$ be a nilpotent element of $\g$.
 By \cite{Ara09b},
 $$X_{H_{DS,f}^0(\L_k(\g))}=X_{\L_k(\g)}\cap \mc{S}_f,$$
 where
 $\mc{S}_f$ is the Slodowy slice at $f$ in $\g$.
 Therefore, in the case  that $k$ is admissible,
$X_{H_{DS,f}^0(\L_k(\g))}=\overline{\mathbb{O}_k}\cap \mc{S}_f$  is a  {\em nilpotent Slodowy slice},
that is,
the intersection of the Slodowy slice with a nilpotent orbit closure,
provided that
$H_{DS,f}^0(\L_k(\g))\ne 0$,
or equivalently,
$\overline{\mathbb{O}_k}\cap \mc{S}_f\ne \emptyset$,
that is,
$f\in \overline{\mathbb{O}_k}$.
In particular,
$H_{DS,f}^0(\L_k(\g))$ is quasi-lisse if 
$f\in \overline{\mathbb{O}_k}$,
and so is the simple $W$-algebra
$\W_k(\g,f)$.
If $f\in \mathbb{O}_k$,
then $\overline{\mathbb{O}_k}\cap \mc{S}_f$ is a point by the transversality of the Slodowy slices,
and therefore,
$\W_k(\g,f)$ is lisse.

The good grading \eqref{eq:grading}
is called {\em even} of $\g_j=0$ unless $j\in \Z$.
If this is the case,
$\W^k(\g,f)$ is $\Z_{\geq 0}$-graded  and thus of CFT type.


The following assertion was stated in the case that $f\in \mathbb{O}_k$ in \cite{AEkeren19},
but the same proof applies.
\begin{Th}
Let $k$ be admissible and $f\in \overline{\mathbb{O}}_k$.
Suppose that $f$ admits a good even graing.
Then, for $\lam\in Adm^k_{\Z}$,
$\Si{k}{\lam,f}=H_{DS,f}^0(\Si{k}{\lam})$ is simple.
In particular,
$H_{DS,f}^0(\L_k(\g))=\W_k(\g,f)$.
\end{Th}

 \begin{Lem}
 Let $f$ be an even nilpotent element.
 Then 
 $\W_k(\g,f)$ is  self-dual and of CFT type
 with respect to the 
 Dynkin grading.
 \end{Lem}
 \begin{proof}
 Since $f$ is an even nilpotent element,
 the Dynkin grading is even, and so  $\W_k(\g,f)$ 
 is of CFT type.
 The self-duality follows from the formula in \cite[Proposition 6.1]{AEkeren19}.
 \end{proof}

\begin{Rem}
If the grading is not Dynkin,
$\W_k(f,g)$ need not be self-dual, see \cite[Proposition 6.3]{AEkeren19}.
\end{Rem}
\begin{Th}\label{Th:tensor-cateogry}\label{Th::tensor-cateogry}
Let $\g$ be simply-laced,
 $k$  admissible,
and let $f$ be an even   
nilpotent element in 
$\overline{\mathbb{O}}_k$.
Suppose Conjecture \ref{Conj:verte-tensor} is true
for 
$\W_{k+1}(\g,f)$ and also that $\W_{k+1}(\g,f)$ is self-dual.
Then
the functor 
$$ \mc{C}_{\L_k(\g)}^{ord}\ra \mc{C}_{\W_k(\g,f)}^{ord},
\quad M\mapsto H_{DS,f}^0(M),$$
is a unital braided tensor functor.
In particular,
the modules 
 $\Si{k}{\lam,f}$, $\lam \in Adm^k_{\Z}$,
 are rigid.
\end{Th}

\begin{Rem}
Note that by Theorem \ref{Th:decom-adm} we have that $\W_k(\g,f) \otimes \L_1(\g)$ is an extension of $\W_{k+1}(\g,f) \otimes \W_\ell(\g)$, where $\W_\ell(\g)$ is rational and lisse if $k$ is admissible. Especially every ordinary module of $\W_k(\g,f) \otimes \L_1(\g)$ is an object in the category of ordinary modules of $\W_{k+1}(\g,f) \otimes \W_\ell(\g)$ and so if the latter has vertex tensor category structure, then so does the former by \cite{CKM}. In other words, if Conjecture \ref{Conj:verte-tensor} is true for $\W_{k+1}(\g,f)$ and if $k$ is admissible then Conjecture \ref{Conj:verte-tensor} is also true for $\W_{k}(\g,f)$. 

Moreover, \cite[Thm. 5.12]{CKM2} applies to our setting if $\W_{k}(\g,f)$ is $\mathbb Z$-graded.  Thus, if the category of ordinary modules of $\W_{k+1}(\g,f) \otimes \W_\ell(\g)$ is a fusion category, then so is the one of $\W_k(\g,f) \otimes \L_1(\g)$. If $\W_{k}(\g,f)$ is $\mathbb Z$-graded by conformal weight 
 it thus also follows that $\W_{k}(\g,f)$ is rational provided that $\W_{k+1}(\g,f)$ is rational and lisse. 
\end{Rem}

Note that if $f\in \mathbb{O}_k$
then 
Conjecture \ref{Conj:verte-tensor} holds
since
 $\W_k(\g,f)$  is lisse.

In the case that $f=f_{prin}$,
$f\in \overline{\mathbb{O}}_k$ if and only if $\overline{\mathbb{O}}_k=\mc{N}$,
 the nilpotent cone of $\g$.
An admissible number $k$
such that $\overline{\mathbb{O}}_k=\mc{N}$ is called 
 {\em non-degenerate}.
If this is the case
$\W^k(\g)$ is rational,
and the complete fusion rule
$\W_k(\g,f)$
has been  determined previously in
 \cite{FKW92, AEkeren, Cre19}.

In the case that 
$\mathbb{O}_k$ is a subregular nilpotent orbit and $f\in \mathbb{O}_k$,
then
$\W_k(\g,f)$ is rational \cite{AEkeren19},
and
 the fusion rules of $\W_k(\g,f)$  has been determined in \cite{AEkeren19}. 
  
 The following assertion, which follows immediately
 from Theorem \ref{Th::tensor-cateogry},
is new except for type $A$ (\cite{AEkeren19}) and the above cases
since the conjectural rationality \cite{KacWak08,Ara09b} of 
$\W_k(\g,f)$ with  $f\in \mathbb{O}_k$
is    open  otherwise.
\begin{Co}\label{Corollary:rigidity}
Let $\g$ be simply-laced,
 $k$  admissible,
 and suppose that 
 $\mathbb{O}_k$ is an even nilpotent orbit,
 $f\in \mathbb{O}_k$.
 Then $\Si{k}{\lam,f}$ is rigid for all  $\lam\in Adm^k_{\Z}$.
\end{Co}

\begin{Rem}
Let $ \mc{C}_{\W_k(\g,f)}^{KL}$ denote the fusion category
consisting of objects $H_{DS,f}^0(M)$, $M\in \mc{C}_{\L_k(\g)}^{ord}$.
By Theorem \ref{Th:tensor-cateogry},
$$\mc{C}_{\L_k(\g)}^{ord}\ra  \mc{C}_{\W_k(\g,f)}^{KL},\quad
M\mapsto H_{DS,f}^0(M), $$
is a quotient functor between fusion categories.
It gives  an equivalence if and only if the modules
$\Si{k}{\lam,f}$,
$ \lam \in Adm^k_{\Z}$, are distinct.
\end{Rem}

The rest of this section is devoted to the proof 
of Theorem \ref{Th::tensor-cateogry}.

Suppose Conjecture \ref{Conj:verte-tensor} holds for $V$
and
let
$\mathcal C$ be 
a  monoidal full subcategory of $\mc{C}_V^{ord}$.
Recall that
$M, N\in \mc{C}$ are said to {\em centralize each other} \cite{Mu03} 
 if the monodromy of $M$ and $N$ is trivial, i.e., is equal to the identity on $M\boxtimes_V N$,
 where the monodromy is the double braiding 
 $$ M\boxtimes_V N \xrightarrow{b_{M, N}} N \boxtimes_V M \xrightarrow{b_{N, M}} M \boxtimes_V N.$$
 Suppose that 
 $V$ is a vertex operator subalgebra of another quasi-lisse vertex operator algebra $W$
 and that $W$ as a $V$-module is an object of $\mc{C}_V^{ord}$.
 We assume that  Conjecture \ref{Conj:verte-tensor} holds for $W$ as well.
 Let $\mc{D}$ be a monoidal full subcategory of $\mathcal C$
 such that
$W$ and  any object in $\mc{D}$ centralize each other.
Then 
\begin{align*}
\mc{F}(M):=W\boxtimes_V M
\end{align*}
can be equipped with a structure of a module for the \voa{} $W$, 
giving rise to the {\em induction functor} \cite{CKM}
\begin{align*}
\mc{D}\ra C_W^{ord},\quad M\mapsto \mc{F}(M),
\end{align*}

\begin{Th}[\cite{Cre19}]\label{Th:centralize}
Let $\g$ be simply-laced,
and 
let $\lam\in Adm^k_{\Z}$, $\mu\in Adm^{\check{k}}_{\Z}$.
We have
$\mathbf{L}^k_{[\lam,0]}\boxtimes \mathbf{L}^k_{[0,\mu]}\cong \mathbf{L}^k_{[\lam,\mu]}$.
Moreover,
$\mathbf{L}^k_{[\lam,0]}$ and $\mathbf{L}^{\check{k}}_{[0,\mu]}$  centralize each other if $\lam \in 
Adm^k_{\Z}\cap Q$.
\end{Th}

Let $f$ be admissible, and let 
 $f$ be an even nilpotent element in 
$\overline{\mathbb{O}_k}$.
Then $f\in \overline{\mathbb{O}_{k+1}}$ as well
since $\overline{\mathbb{O}_k}$ depends only on the denominator of $k$.
Let $\ell$ be the number defined by
$$\frac{1}{k+h^{\vee}+1}+\frac{1}{\ell+h^{\vee}}=1,$$
that is,
\begin{align*}
\ell+h^{\vee}=\frac{k+h^{\vee}+1}{k+h^{\vee}}.
\end{align*}
Then $\ell$ is a non-degenerate admissible number.
Note that
\begin{align}
Adm^{\ell}_{\Z}=Adm^{k+1}_{\Z}=P_+^{p+q-h^{\vee}},
\quad 
Adm^{\check{\ell}}_{\Z}=Adm^{k}_{\Z}=P_+^{p-h^{\vee}},
\end{align}
where  we have put $k+h^{\vee}=p/q$.

Consider
the $\W_{\ell}(\g)$-modules
$$\mathbf{L}^{\ell}_{[0,\mu]}\cong \mathbf{L}^{\check{\ell}}_{[\mu,0]}
=H_{DS,f_{prin}}^0(\Si{\check{\ell}}{\mu}),\quad \mu\in Adm_{\Z}^{\ell}.$$
Since the stabilizer of  $0\in Adm^{\ell}_{\Z}$
of the $\tilde{W}_+$-action is trivial,
 the simple $\W_{\ell}(\g)$-modules 
 $\mathbf{L}^{\ell}_{[0,\mu]}$, 
 $\mu\in Adm_{\Z}^{\ell}$,
 are distinct.
Therefore,
by Theorem \ref{Th:tensor-cateogry} (that is proved for $f=f_{prin}$ in \cite{Cre19}),
the modules 
$\mathbf{L}^{\ell}_{[0,\mu]}$, 
 $\mu\in Adm_{\Z}^{\ell}$,
form a fusion full subcategory of $\mc{C}^{ord}_{\W_{\ell}(\g)}$
that is equivalent to a  category that can be called a {\emph simple current twist} of
$\mc{C}_{\L_{\check{\ell}}(\g)}^{ord}\cong \mc{C}_{\L_k(\g)}^{ord}$ (see \cite[Thm.7.1]{Cre19} for the details).
This simple current twist of $\mc{C}_{\L_k(\g)}^{ord}$ is the fusion subcategory of $\mc{C}_{\L_k(\g)}^{ord} \boxtimes \mc{C}_{\L_1(\g)}^{ord}$
whose simple objects are the $\Si{k}{\mu,f}\* \Si{1}{-\mu+Q}$. Call this category $\mc{C}_{\L_k(\g)}^{ord, tw}$.

Now set
\begin{align*}
W:=\W_k(\g,f)\* \L_1(\g)=\Si{k}{0,f}\* \Si{1}{\nu},
\quad 
V:=\W_{k+1}(\g,f)\* \W_{\ell}(\g)=\Si{k+1}{0,f}\* \mathbf{L}_{[0,0]}.
\end{align*}
By Theorem \ref{Th:decom-adm},
\begin{align}\label{eq:dec-of-W}
W \cong \bigoplus_{\substack{\lam \in Adm^{k+1}_{\Z}\cap Q
}} \Si{k+1}{\lam,f} \otimes \mathbf{L}^{\ell}_{[\lam,0]}.
\end{align}
Hence,
$V$ is a vertex subalgebra of $W$.
Moreover,
each direct summand of $W$
is an ordinary  $V$-module and the sum is finite,
and so
$W$ is an object of $\mc{C}_V^{ord}$.

Clearly, the $V$-modules
$$\Si{k+1}{0,f}\* \mathbf{L}_{[0,\mu]}^{\ell}
=\W_{k+1}(\g,f)\* \mathbf{L}_{[0,\mu]}^{\ell},\quad \mu \in Adm^{\check{\ell}}_{\Z},$$
form 
 a monoidal full subcategory of $\mc{C}^{ord}_V$
that  is 
equivalent to 
$ \mc{C}_{\L_k(\g)}^{ord}$.
We denote by $\mc{D}$
this fusion category.

By  Theorem \ref{Th:centralize} and
\eqref{eq:dec-of-W},
we find  that
$W$ centralizes 
any object of $\mc{D}$.
Hence,
we have the induction functor
\begin{align*}
\mc{F}: \mc{D}\ra \mc{C}^{ord}_W,\quad 
M\mapsto W\boxtimes_{V}M.
\end{align*}

\begin{Th}\label{Th:equivalence}
The induction functor 
$\mc{F}: \mc{D}\ra \mc{C}^{ord}_W$
is a fully faithful tensor functor
that sends 
$\Si{k+1}{0,f}\* \mathbf{L}_{[0,\mu]}^{\ell}$ to 
$\Si{k}{\mu,f}\* \Si{1}{-\mu+Q}$,
where 
 $-\mu+Q$ denotes the class in $ P/Q\cong P_+^1$.
\end{Th}
\begin{proof}
Thanks to \cite[Theorem 3.5]{Cre19},
it is sufficient to show that
$\mc{F}(\Si{k+1}{0,f}\* \mathbf{L}_{[0,\mu]}^{\ell})$
is a simple
$W$-module
that is isomorphic to 
$\Si{k}{\mu,f}\* \Si{1}{-\mu+Q}$
 for all  $\mu \in Adm^{\check{\ell}}_{\Z}$.
 As $V$-modules we have
 \begin{align*}
\mc{F}(\Si{k+1}{0,f}\* \mathbf{L}_{[0,\mu]}^{\ell})
= \bigoplus_{\substack{\lam \in Adm^{k+1}_{\Z}\cap Q
}} (\Si{k+1}{\lam,f} \otimes \mathbf{L}^{\ell}_{[\lam,0]})
\boxtimes_V (\Si{k+1}{0,f}\* \mathbf{L}_{[0,\mu]}^{\ell})\\
\cong  \bigoplus_{\substack{\lam \in Adm^{k+1}_{\Z}\cap Q
}}\Si{k+1}{\lam,f} \* \mathbf{L}_{[\lam,\mu]}^{\ell}\cong 
\Si{k}{\mu,f}\* \Si{1}{-\mu+Q}
\end{align*}
by Theorem \ref{Th:decom-adm}
and Theorem \ref{Th:centralize}.
 We claim this is indeed an isomorphism of $W$-modules.
 To see this,
 it is sufficient to show that 
 there is a non-trivial $W$-module homomorphism
 $\mc{F}(\Si{k+1}{0,f}\* \mathbf{L}_{[0,\mu]}^{\ell})
\ra \Si{k}{\mu,f}\* \Si{1}{-\mu+Q}$
since $\Si{k}{\mu,f}\* \Si{1}{-\mu+Q}$ is simple.
By the Frobenius reciprocity \cite{KO, CKM},
we have
\begin{align*}
&\Hom_{W\on{-Mod}}(\mc{F}(\Si{k+1}{0,f}\* \mathbf{L}_{[0,\mu]}^{\ell}),
\Si{k}{\mu,f}\* \Si{1}{-\mu+Q})\\
&\cong 
\Hom_{V\on{-Mod}}(\Si{k+1}{0,f}\* \mathbf{L}_{[0,\mu]}^{\ell},\Si{k}{\mu,f}\* \Si{1}{-\mu+Q})
\\
&\cong 
\bigoplus_{\substack{\lam \in Adm^{k+1}_{\Z}\cap Q
}}
\Hom_{V\on{-Mod}}(\Si{k+1}{0,f}\* \mathbf{L}_{[0,\mu]}^{\ell},\Si{k+1}{\lam,f} \* \mathbf{L}_{[\lam,\mu]}^{\ell}).
\end{align*}
It follows that
there is a non-trivial homomorphism
 corresponding to
 the identity map 
 $\Si{k+1}{0,f}\* \mathbf{L}_{[0,\mu]}^{\ell}\ra \Si{k+1}{0,f}\* \mathbf{L}_{[0,\mu]}^{\ell}$.
\end{proof}

\begin{proof}[Proof of Theorem \ref{Th:tensor-cateogry}]
By Theorem \ref{Th:equivalence},
the correspondence
\begin{align*}
\Si{k}{\mu}\* \Si{1}{-\mu+Q}\mapsto \mathbf{L}^{\ell}_{[0,\mu]}\mapsto \Si{k+1}{0,f}\*\mathbf{L}^{\ell}_{[0,\mu]}
\mapsto \Si{k}{\mu,f}\* \Si{1}{-\mu+Q},
\quad \mu\in Adm^k_{\Z}=Adm^{\check{\ell}}_{\Z},
\end{align*}
gives 
a tensor functor 
\begin{align*}
\mc{C}_{\L_k(\g)}^{ord, tw}\ra C_{\W_k(\g,f)\* \L_1(\g)}^{ord}
=C_{\W_k(\g,f)}\boxtimes \mc{C}_{\L_1(\g)}^{ord},
\end{align*}
where the last $\boxtimes$ denotes the Deligne product. This functor is fully faithful if and only if all the $\Si{k}{\mu,f}$ are non-isomorphic.
Since $\mc{C}_{\L_1(\g)}^{ord}$ is semi-simple and its simple objects are all invertible, i.e., simple currents, this corresponds extends to a surjective tensor functor \begin{align*}
\mc{C}_{\L_k(\g)}^{ord} \boxtimes \mc{C}_{\L_1(\g)}^{ord} \ra C_{\W_k(\g,f)}\boxtimes \mc{C}_{\L_1(\g)}^{ord},
\end{align*}
which restricts to a surjective tensor functor 
\begin{align*}
\mc{C}_{\L_k(\g)}^{ord}  \ra C_{\W_k(\g,f)},\qquad \Si{k}{\mu}\ra \Si{k}{\mu,f}=H_{DS,f}^0(\Si{k}{\mu})
\end{align*}
that again is fully faithful if and only if all the $\Si{k}{\mu,f}$ are non-isomorphic.
\end{proof}

Let us conclude with a remark on the general case.
\begin{Rem}
It is a well-known result of Kazhdan-Lusztig that ordinary modules of affine vertex algebras at generic level have vertex tensor category \cite{KL1,KL2,KL3,KL4} and it is reasonable to expect that a similar result might hold for $\W$-algebras as well. However this is only proven for the Virasoro algebra \cite{CJORY} and the $N=1$ super Virasoro algebra \cite{CMY}.

At generic levels one has to deal with infinite order extensions of vertex algebras and so one needs to consider completions of vertex tensor categories.
The theory of vertex algebra extension also works for such completions \cite{CMY}. This means that if one can prove the existence of vertex tensor category structure on categories of ordinary modules of $\W$-algebras at generic levels, then one can also derive results for generic levels that are similar to the statements of this section. 
\end{Rem}



\end{document}